\documentclass[11pt]{amsart}

\usepackage{amscd, amsfonts, amsmath, amssymb, amsthm, mathrsfs, appendix, enumerate, graphicx, flexisym, epsfig, float, tabu, graphicx, hyperref, pdfpages, tikz, mathtools, comment, tikz-cd,pinlabel,bbm}
\usepackage[all,cmtip]{xy}
\makeatletter
\@namedef{subjclassname@2020}{Mathematics Subject Classification \textup{2020}}

\newsavebox{\@brx}
\newcommand{\llangle}[1][]{\savebox{\@brx}{\(\m@th{#1\langle}\)}%
	\mathopen{\copy\@brx\kern-0.5\wd\@brx\usebox{\@brx}}}
\newcommand{\rrangle}[1][]{\savebox{\@brx}{\(\m@th{#1\rangle}\)}%
	\mathclose{\copy\@brx\kern-0.5\wd\@brx\usebox{\@brx}}}

\makeatother

\newtheorem{theorem}{Theorem}[section]
\newtheorem{corollary}[theorem]{Corollary}

\newtheorem{prop}[theorem]{Proposition}

\theoremstyle{definition}

\newtheorem{defn}[theorem]{Definition}
\newtheorem{example}[theorem]{Example}
\newtheorem{remark}[theorem]{Remark}
\numberwithin{equation}{subsection}
\newtheorem*{ack}{Acknowledgement}

\newcommand{\Aut}{\operatorname{Aut}}

\newcommand{\GLcat}{\mathbf{GL}}

\newcommand{\Conj}{\operatorname{Conj}}

\newcommand{\C}{\mathbf{C}}

\newcommand{\Ab}{\mathbf{Ab}}
\newcommand{\T}{\mathrm{T}}

\newcommand{\Sa}{\textrm{S}}

\newcommand{\Hom}{\operatorname{Hom}}
\newcommand{\GL}{\operatorname{GL}}

\newcommand{\U}{\mathtt{u}}
\newcommand{\D}{\mathtt{d}}

\newcommand{\id}{\mathrm{id}}

\newcommand{\ie}{\mathrm{i.e.}}

\begin{document}
	\title[Generalised Legendrian racks]{Generalised Legendrian racks of Legendrian links}
	\author{Biswadeep Karmakar}	
	\author{Deepanshi Saraf}
	\author{Mahender Singh}
	
	\address{Department of Mathematical Sciences, Indian Institute of Science Education and Research (IISER) Mohali, Sector 81,  S. A. S. Nagar, P. O. Manauli, Punjab 140306, India.}
	\email{biswadeep.isi@gmail.com}
	\email{saraf.deepanshi@gmail.com}
	\email{mahender@iisermohali.ac.in}
	
	

\begin{abstract}
A generalised Legendrian rack is a rack equipped with a Legendrian structure, which is a pair of maps encoding the information of Legendrian Reidemeister moves together with up and down cusps in the front diagram of an oriented Legendrian link. Employing a purely rack theoretic approach, we associate a generalised Legendrian rack (or a $\GL$-rack) to an oriented Legendrian link, and prove that it is an invariant under Legendrian isotopy. As immediate applications, we prove that this invariant distinguishes infinitely many oriented Legendrian unknots and oriented Legendrian trefoils. To comprehend their algebraic structure, we prove that every $\GL$-rack admits a homogeneous representation. Further, using the idea of trunks, we define modules over $\GL$-racks, and prove the equivalence of the category of $\GL$-rack modules and the category of Beck modules over a fixed $\GL$-rack.
\end{abstract}
\maketitle

\section{Introduction}
Racks are algebraic structures equipped with distributive binary operations. These objects are used for constructing invariants of framed topological links in $\mathbb{R}^3$, and they also give non-degenerate solutions of the Yang-Baxter equation. Quandles are stronger avatars of racks, which encode algebraically the three Reidemeister moves of planar diagrams of oriented topological links in $\mathbb{R}^3$. The study of racks and quandles gained thrust after the pioneering works of Joyce \cite{Joyce1979} and Matveev \cite{MR0672410}, who showed that fundamental quandles are complete invariants of non-split links up to orientation of the ambient space. Although fundamental quandles are strong invariants of oriented topological links, the isomorphism problem for them is difficult. This motivated search for invariants of quandles themselves, most notably, (co)homological invariants.
	\par
	
	A (co)homology theory for racks was introduced in \cite{MR1364012, MR2255194} as homology of the associated rack space of a rack. This led to the development of a (co)homology theory for quandles in \cite{MR1990571}, where state-sum invariants using quandle cocycles as weights were defined for oriented topological links. Giving a new perspective to quandle cohomology in terms of model categories, the recent work \cite{MR3937311} showed that quandle cohomology is a Quillen cohomology.
	\par
	
	There are various generalisations and refinements of classical knot theory. In this paper, we consider a geometrically motivated refinement, namely, the Legendrian knot theory. We equip $\mathbb{R}^3$ with the standard contact structure, which is the kernel of the contact form $dz-ydx$. Then, a Legendrian link is a smooth link in $\mathbb{R}^3$ such that it is tangent to the standard contact structure at each of its points. Analogous to topological links, Legendrian links are studied upto Legendrian isotopy and via appropriate planar projections called front projections. Each Legendrian link admits a front projection with finitely many crossings and cusps. An analogue of the classical Reidemeister Theorem is also known in the Legendrian setting, which makes the diagram based study of these objects feasible. Many invariants of Legendrian links are already known in the literature. See, for example \cite{MR1946550, MR1957049}, and the survey article \cite{MR2179261} for more details. It turns out that Legendrian isotopic Legendrian links are smoothly isotopic, but smoothly isotopic Legendrian links may not be Legendrian isotopic. For instance, there are infinitely many Legendrian knots (up to Legendrian isotopy) that are smoothly isotopic to the unknot. In fact, a complete classification of Legendrian unknots is already known \cite{MR2496415, MR1619122}. Thus, the classification problem of Legendrian links may be considered as a refinement of the classification problem of topological links.
	\par

It is intriguing to know whether there are analogues of racks and related (co)homology theories for studying Legendrian links. Furthermore, it is natural to ask whether such an invariant for Legendrian links recovers the corresponding known invariant for underlying topological links. We attempt to address these questions by introducing generalised Legendrian racks (or simply $\GL$-racks), which, informally speaking, are racks together with a pair of maps that encode the information of Legendrian Reidemeister moves along with the information of upward and downward cusps of an oriented front diagram. Our construction strengthens and enhances the preceding rack-theoretic invariants appearing in \cite{MR4292403} and \cite{arXiv:1706.07626}. Note that, the rack invariant considered in \cite{MR4292403} already recovers the one appearing in \cite{arXiv:1706.07626}, and that \cite{MR4292403} is focussed on colouring invariants obtained from finite Legendrian racks. Our notion of a $\GL$-rack of an oriented Legendrian link also appears to be an appropriate refinement of the fundamental quandle of the underlying oriented topological link. We shall see that bi-Legendrian rack introduced in \cite{MR4586264} is equivalent to the $\GL$-rack that we introduce here. Further, the focus of \cite{MR4586264} is again coloring invariants of Legendrian knots. As an application of our construction, we are able to immediately distinguish infinitely many oriented Legendrian unknots and trefoils. We then proceed to explore the algebraic and categorical properties of $\GL$-racks. Further, we also identify the appropriate coefficient objects for defining (co)homology theories for $\GL$-racks. Our approach is purely rack theoretic, and does not incorporate any geometric ideas from the theory of Legendrian links.
	\par
	
The paper is organised as follows. In Section \ref{sec legendrian links}, we recall some basics from Legendrian knot theory. In Section \ref{sec generalised legendrian racks}, we define a generalised Legendrian rack, and give many examples of these structures. In Section \ref{generalised Legendrian rack of link}, we define the generalised Legendrian rack of an oriented Legendrian link, and prove that it is invariant under the Legendrian equivalence (Theorem \ref{generalised Legendrian rack invariant}). As an application, we deduce that this invariant distinguishes infinitely many oriented Legendrian unknots (Theorem \ref{unknots distinguish}), and infinitely many oriented Legendrian trefoils (Theorem \ref{trefoils distinguish}). Furthermore, it is proved that two topologically distinct non-split oriented Legendrian links cannot have isomorphic generalised Legendrian racks (Theorem \ref{legendrian link distinguishes distinct top links}). We also discuss the idea of colouring of an oriented Legendrian link by a generalised Legendrian rack, and deduce that it gives an invariant of Legendrian links (Corollary \ref{cor generalised Legendrian coloring}). Adapting the idea of Joyce to the Legendrian setting, in Section \ref{homogenous representation}, we prove that every generalised Legendrian rack admits a homogeneous representation as a generalised Legendrian rack which is built up with a disjoint union of cosets of a family of subgroups of some appropriate group (Theorem \ref{thm homogenous representation}). In Section \ref{sec GL modules}, we define appropriate modules over $\GL$-racks, which could be used for introducing (co)homology theories for $\GL$-racks, and give some examples. In Theorem \ref{equivalence of categories for symmetric rack modules} of Section \ref{sec Category of modules over GL-racks}, we prove the equivalence of the category of $\GL$-rack modules and the category of Beck modules over a fixed $\GL$-rack.
	\medskip
	
	\section{Contact structure and Legendrian links}\label{sec legendrian links}
	We begin by recalling some basic ideas from contact topology that are essential for our discussions. The reader may refer to \cite{MR1648083, MR2024631, MR2397738} for finer details. We restrict ourselves to 3-manifolds, specifically $\mathbb{R}^3$, and refer the reader to the excellent survey article \cite{MR2179261} on Legendrian knots.
	\par

	\begin{defn}
		A 1-form $\alpha$ on an oriented 3-dimensional manifold $M$ is called a {\it contact form} if $\alpha_x \wedge (d \alpha_x) \neq 0$ for all $x \in M$. A hyperplane distribution $\xi$ is a {\it contact structure} on $M$ if it is the kernel of a contact form, that is,  $\xi=\ker(\alpha)$ for a contact 1-form $\alpha$ on $M$. A manifold $M$ equipped with a contact structure $\xi$ is called a {\it contact manifold}, and denoted as $(M, \xi)$.
	\end{defn}
	
	Here, $\xi$ being a hyperplane distribution means that $\xi_x$ is a two dimensional subspace of the tangent space $T_xM$ at each $x\in M$. We also require that $\alpha \wedge d \alpha$ defines the given orientation on $M$. The contact condition  $\alpha \wedge d \alpha \ne 0$ is a {\it total non-integrability} condition, which ensures that there is no embedded surface in $M$ which is tangent to $\xi$ everywhere. 
	\par

	\begin{example}
		The simplest contact structure on $\mathbb{R}^3$ (with respect to the Cartesian coordinates $(x, y, z)$) is given by
		$$\xi^{std} = \textrm{span} ~\left\{ \frac{\partial}{\partial y},~ \frac{\partial}{\partial x} +y\frac{\partial}{\partial z} \right\},$$
		which is the kernel of the contact form $$\alpha_0= dz-ydx.$$  The contact structure $\xi^{std}$ is called the {\it standard contact structure} on $\mathbb{R}^3$. Note that the hyperplanes along the $xz$-plane are all horizontal, that is, parallel to the $xy$-plane. Moving along the positive $y$-axis from the origin, the hyperplanes start to twist around the $y$-axis in a left handed manner with the planes making a $90^\circ$ twist at infinity (see Figure \ref{standard contact st}, taken from \cite[Section 2.1]{MR2179261}).
		\begin{figure}[h]
			\centering
			\includegraphics[width=.5\textwidth]{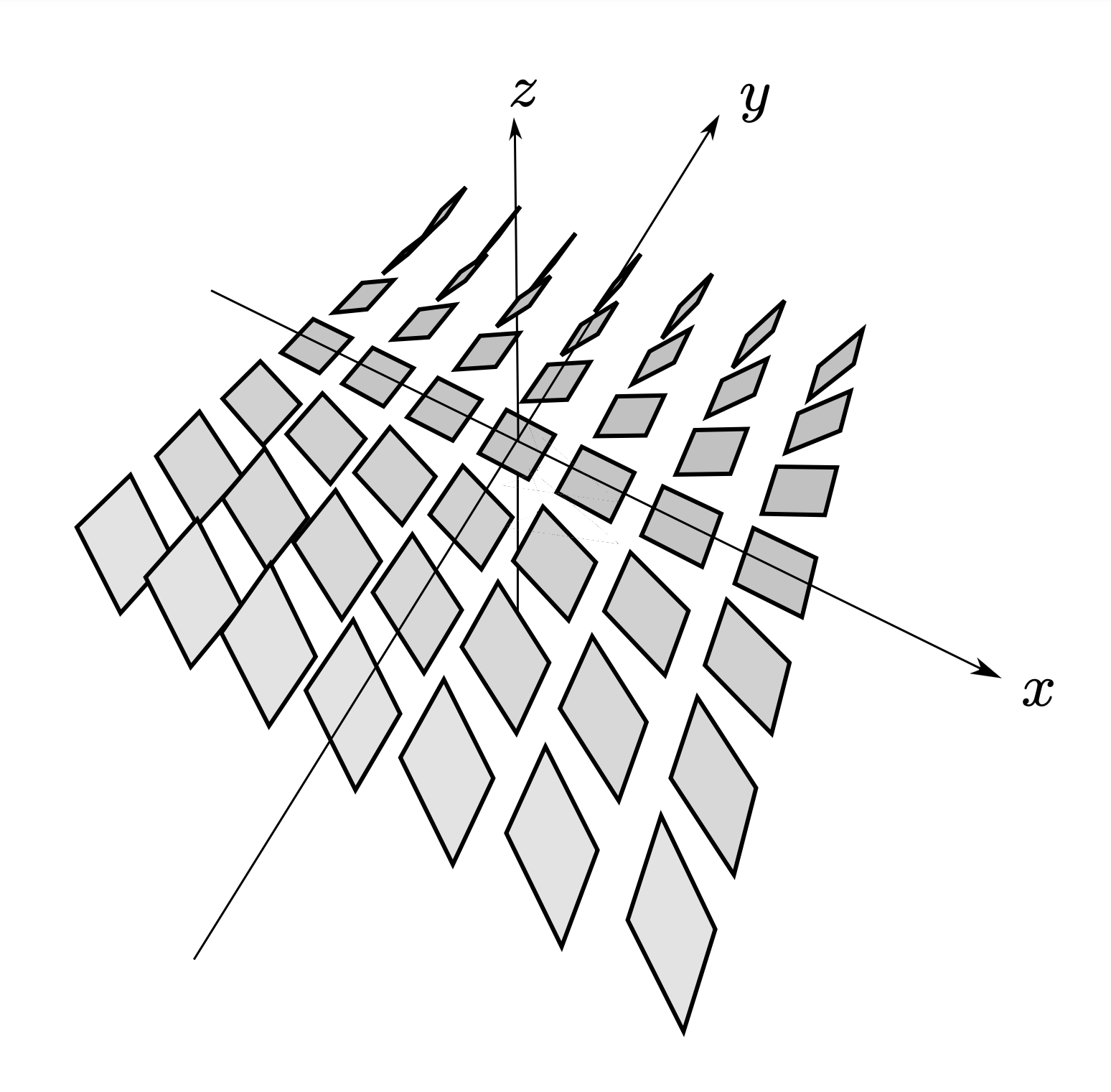}
			\caption{Standard contact structure on  $\mathbb{R}^3$.}\label{standard contact st}
		\end{figure}
	\end{example}
	
	There are many contact structures on a 3-manifold $M$. But, by Darboux's Theorem, any point in $M$ has an open neighborhood that is diffeomorphic to an open neighborhood of the origin in  $\mathbb{R}^3$ by a diffeomorphism that takes the contact structure on $M$ to the standard contact structure $\xi^{std}$ on $\mathbb{R}^3$. Throughout the article, the contact form $\alpha_0 $ and the contact structure $\xi^{std} $ on $\mathbb{R}^3$ are fixed.

	\begin{defn}
		A  {\it Legendrian knot} in  $(\mathbb{R}^3,\xi^{std}) $ is a smooth embedding $\gamma:\mathbb{S}^1 \rightarrow \mathbb{R}^3$ such that  $\gamma'(\theta) \in \xi^{std}_{\gamma(\theta)}  \text{ for all } \theta \in \mathbb{S}^1 .$
	\end{defn}
	
	As in topological knot theory, planar projections play a crucial role in Legendrian knot theory. Let $\gamma(\theta)=(x(\theta),y(\theta),z(\theta))$ be a parametrized Legendrian knot in $(\mathbb{R}^3,\xi^{std})$. Then the {\it front projection} of $\gamma$ is the map $\pi: \mathbb{S}^1 \rightarrow \mathbb{R}^2$ given by $\pi(\theta)=(x(\theta),z(\theta))$. A parametrization of a Legendrian  knot will induce an orientation on the knot, and hence we consider oriented Legendrian knots by choosing this parametrization induced orientation.
	\par
	
	Note that a parametrized Legendrian knot $\gamma $ satisfies the equation $z'(\theta) = y(\theta)x'(\theta)$. There is no point on the front projection where the tangent is vertical, instead we have cusps on both right and left side. At cusp points, we have $x'(\theta)=z'(\theta)=0$ and $y=0 $. Moving away from cusps, the $y$-coordinate can be recovered as $y(\theta)=\frac{z'(\theta)}{x'(\theta)}$. We note that, at a crossing in the front projection, the strand with the lower slope will always be above. For example, see Figure \ref{Legendrian trefoil} for the front projection of a Legendrian trefoil.  A {\it Legendrian link} can be defined analogously by replacing $\mathbb{S}^1$ with the disjoint union $\sqcup_{i=1}^n~ \mathbb{S}^1$.
	
	\begin{figure}[h] 
		\centering
		\includegraphics[width=.25\textwidth]{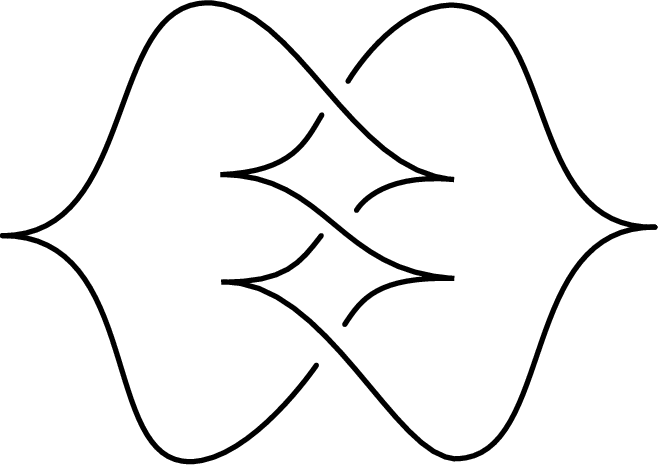}
		\caption{A Legendrian trefoil.}\label{Legendrian trefoil}
	\end{figure}
	
	\begin{defn}
		A {\it front diagram} is an immersion of a disjoint union of circles in the $xz$-plane such that the following hold:
		\begin{enumerate}
			\item The immersion is an embedding at all but finitely many points. At these points, there is a cusp or a double point.
			\item There is no vertical tangency at any point.
		\end{enumerate}
	\end{defn}
	
	The planar Legendrian isotopy between two front diagrams is defined in the usual way. We note that every Legendrian link gives rise to a front diagram by taking the front projection. Conversely, given a front diagram, there is a unique Legendrian link whose front projection agrees with this front diagram. 
	
	\begin{defn}\label{def1}
		Two Legendrian links $K$ and $K'$ are called {\it Legendrian equivalent} if there exists a smooth map $F: \mathbb{S}^1\times [0,1] \rightarrow (\mathbb{R}^3, \xi^{std})$ such that the following hold:
		\begin{enumerate}
			\item $F(\mathbb{S}^1 \times \{ 0\})=K$.
			\item $F(\mathbb{S}^1 \times \{ 1\})=K'$.
			\item  $F|_{\mathbb{S}^1 \times \{ t\}}$ is a Legendrian embedding for each $t \in [0, 1]$.
		\end{enumerate}
	\end{defn}
	
	The local moves of a front diagram shown in Figure \ref{Local Legendrian Reidemeister moves} are called {\it Legendrian Reidemeister moves}, written LR moves for short.
	
	\begin{figure}[h]
		\centering
		\includegraphics[width=.5\textwidth]{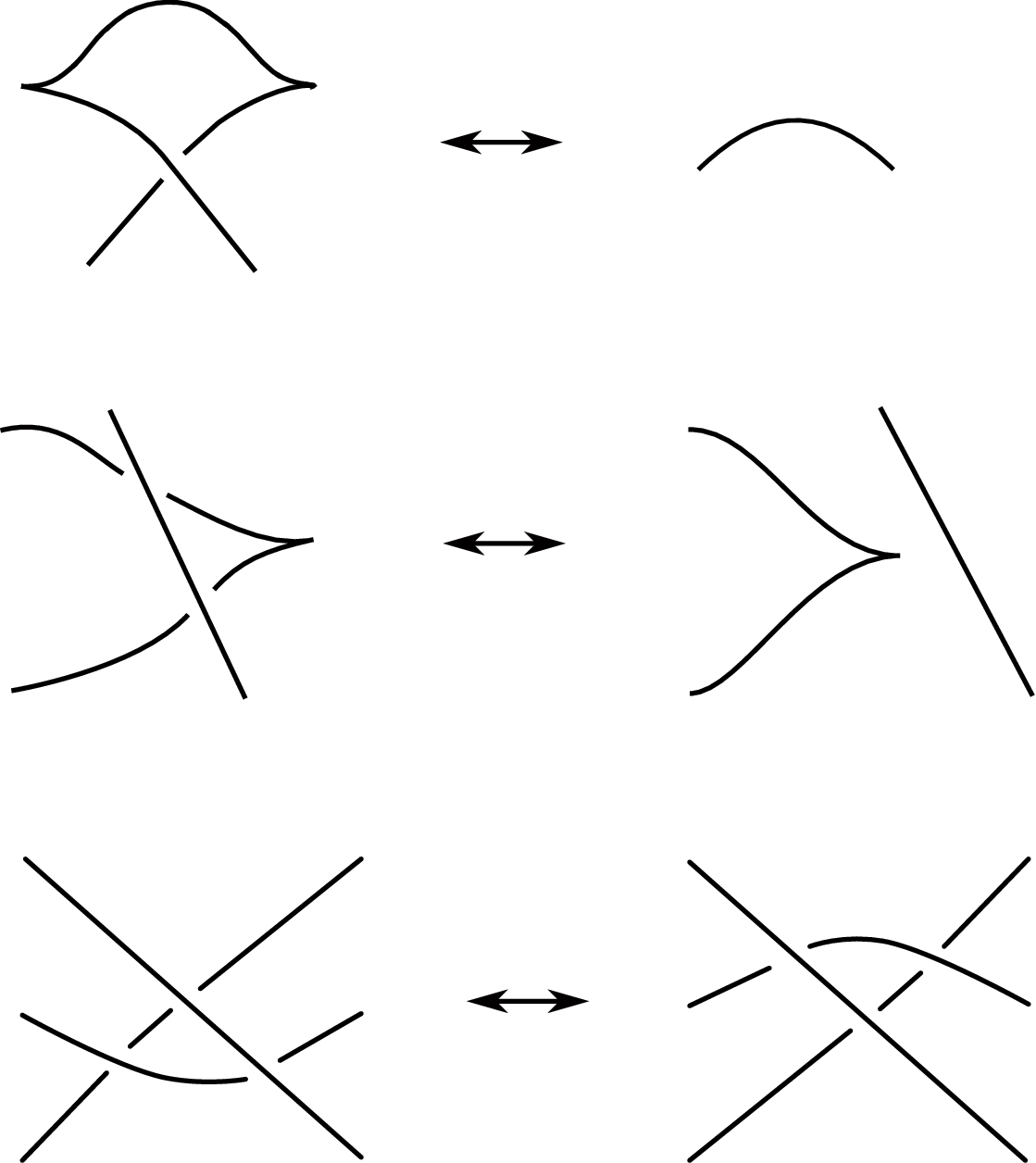}
		\caption{Legendrian Reidemeister moves.}\label{Local Legendrian Reidemeister moves}
	\end{figure}

	An analogue of the Reidemeister Theorem for topological links is known for Legendrian links \cite{MR0911776, MR1186009}.
	
	\begin{theorem}\label{legendrian reidemeister thm}
		Two Legendrian links $K$ and $K'$ are Legendrian equivalent if and only if any front diagram $D(K)$ of $K$ can be obtained from any front diagram $D(K')$ of $K'$ by a finite sequence of Legendrian Reidemeister moves and planar Legendrian isotopy.
	\end{theorem}
	
	In the pursuit of answering the formidable problem of classifying Legendrian links, several invariants have been studied in the literature. We briefly recall a few classical invariants here, and refer the reader to \cite{MR2179261} for a detailed and geometric explanation of the same.
	\par
	The simplest invariant of a Legendrian link is its \textit{topological link type}. The next invariant, called the \textit{Thurston-Bennequin invariant} and denoted by $tb(K)$, can be computed from the oriented front diagram $D(K)$ of $K$ as 
	$$tb(K)=\mathrm{writhe}(\mathit{D(K)})-\frac{1}{2}(\mathrm{number~of~cusps~in~} \mathit{D(K)} ).$$
	Another well-known invariant is the \textit{rotation number} $r(K)$ of the Legendrian link $K$, which is computed  from the oriented front diagram $D(K)$ of $K$ as
	$$r(K)=\frac{1}{2}(D-U),$$ where $D$ is the number of down cusps and $U$ is the number of up cusps in $D(K)$.
	\par
	
	It turns out that these classical invariants, although powerful, are not enough to distinguish several Legendrian links, which led to development of more sophisticated invariants such as Chekanov-Eliashberg DGA \cite{MR1946550}, etc.
	\medskip
	
	\section{generalised Legendrian racks}\label{sec generalised legendrian racks}
	Recall that, a {\it rack} $(X, *)$ is a set $X$ with a binary operation $*$ satisfying the following:
	\begin{enumerate}
		\item Given $x,y\in X$, there exists a unique $z\in X$ such that $z\ast y=x$.
		\item $(x \ast y)\ast z=(x \ast z)\ast{(y\ast z)}$ for all $x,y,z\in X$. 
	\end{enumerate}
	
	A {\it quandle} is a rack with the additional axiom of idempotency, that is, $x*x=x$ for all $x \in X$. For example, each group $G$ turns into a quandle $\Conj(G)$ with respect to the binary operation $x*y= y^{-1} xy$ of conjugation.  A generalised Legendrian rack is a rack with an additional structure and its defining axioms are derived from the oriented Legendrian Reidemeister moves.
	
	\begin{defn}
		A \textit{generalised Legendrian rack} or a $\GL$-rack  is a quadruple $(X, \ast, \U,\D)$, where 	$(X, \ast)$ is a rack and  $\U,\D: X \rightarrow X$ are maps such that the 
		following hold for all $x, y \in X$:
		\begin{enumerate}
			\item[(L1)~] $ \U\D(x \ast x) =x$.
			\item[(L1$^{\prime}$)] $ \D\U(x \ast x) =x$.
			\item[(L2)~] $ \U(x \ast y)= \U(x) \ast y $.
			\item[(L2$^{\prime}$)] $ \D(x \ast y)= \D(x) \ast y $.
			\item[(L3)~] $x  \ast \U(y)= x \ast y$. 
			\item[(L3$^{\prime}$)] $x  \ast \D(y)= x \ast y$. 
		\end{enumerate}	
		We refer the pair $(\U,\D)$ as a \textit{generalised Legendrian structure or a $\GL$-structure} on the rack $(X, *)$.
	\end{defn}
	
	Note that, if $(X, \ast,\U,\D)$ is a $\GL$-rack, then $(X, \ast)$ is a quandle if and only if $\D\U =\id_X= \U \D$. We begin with the following observations.
	
	\begin{prop}\label{u,d an auto}
		Let $(X, \ast,\U,\D)$ be a $\GL$-rack. Then the following hold:
		\begin{enumerate}
			\item The maps $\U$ and $\D$ are automorphisms of the underlying rack $(X, *)$.
			\item $\D\U(x) \ast x=x$ and $\U\D(x) \ast x=x$ for all $x \in X$.
			\item $\D\U(x)=\U\D(x)$ for all $x \in X$.
			\item If the underlying rack $(X, *)$ is involutory, $\ie$, $x \ast y =x \ast^{-1} y$ for all $x,y \in X$, then $(X, \ast,\U^{-1},\D^{-1})$ is also a $\GL$-rack.
		\end{enumerate}
	\end{prop}
	
	\begin{proof}
		Axioms (L2) and (L3) imply that  $\U$ is an endomorphism of $(X, *)$. Axiom (L1) immediately gives surjectivity of $\U$. Finally, axiom (L1$^{\prime}$) together with $\U$ being a homomorphism implies that $\U$ is injective as well. A similar argument shows that $\D$ is also an automorphism of $(X, *)$, which proves assertion (1).
		\par
		Axiom (L1) together with the fact that $\U$ and $\D$ are homomorphisms gives $\U\D(x) \ast \U\D(x) =x$. Now, axioms (L3) and (L3$^\prime$) gives $\U\D(x) \ast x =x$. Similarly, one obtains  $\D\U(x) \ast x =x$,  proving assertion (2), which further yields assertion (3) immdediately.
		\par
		Applying $\U^{-1}$ to $\U(x \ast y)= \U(x) \ast y $ and $x  \ast \U(y)= x \ast y$ shows that $\U^{-1}$ satisfies (L3) and (L2), respectively. Similarly, $\D^{-1}$ satisfies (L3$^\prime$) and (L2$^\prime$). 
		Assertion (2) gives $\D\U(x) \ast x=x$, which implies that $x =\U^{-1}\D^{-1}(x) \ast^{-1} \U^{-1}\D^{-1}(x)$. Since $(X, *)$ is involutory, this gives $x =\U^{-1}\D^{-1}(x) \ast \U^{-1}\D^{-1}(x)$. Further, since 
		$\D$ and $\U$ are homomorphisms, we obtain $x =\U^{-1}\D^{-1}(x \ast x)$. A similar argument yields  $x =\D^{-1}\U^{-1}(x \ast x)$, which establishes assertion (4).
	\end{proof}
	
	\begin{remark}
		Taking $\U=\D$, we obtain the Legendrian rack considered in \cite[Definition 4]{MR4292403}. Furthermore, Proposition \ref{u,d an auto}(2) shows that for $\U=\D$, the definition considered in \cite[Definition 4]{MR4292403} has one redundant axiom.
	\end{remark}
	
	Let $(X, \ast, \U, \D)$ and $(Y, \star, \U^\prime,\D^\prime)$ be two $\GL$-racks. Then, a {\it $\GL$-rack homomorphism} $\phi:(X, \ast, \U, \D) \to (Y, \star, \U^\prime,\D^\prime)$ is a map $\phi:X \to Y$ such that the following hold:
	\begin{enumerate}
		\item $\phi(x\ast y)=\phi(x) \star \phi(y)$.
		\item $\phi ~\U= \U^\prime~\phi $.
		\item $\phi ~\D= \D^\prime~\phi $.
	\end{enumerate}
	
	Recall that, for each $y \in X$, the inner automorphism $S_y:X \to X$ given by $S_y(x)= x*y$ is a rack automorphism of $(X,*)$. Furthermore, axiom (L2) and (L2$^\prime$) imply that $S_y$ is a $\GL$-rack automorphism of $(X, \ast, \U,\D)$ for each $y\in X$. It is easy to see that the set of all $\GL$-racks together with their morphisms forms a subcategory of the category of racks.
	\par
	
	\begin{example}
		Let us look at some examples.
		\begin{enumerate}
			\item If $(X, *)$ is a quandle, then $(X, \ast, \id_X,\id_X)$  is a $\GL$-rack. We refer to such racks as {\it trivial $\GL$-racks}.
			\item Let $G$ be a group and $u, v, w \in G$ three commuting elements such that $uvw =1$. Then the binary operation $x*y=yu y^{-1} x$ ~gives a rack structure on $G$. Further, the maps $\U, \D:G \to G$ given by $\U(x)= xv$ and $\D(x)=xw$ gives a $\GL$-structure on $(X, *)$.
		\end{enumerate}
	\end{example}
	
	We shall see in Section \ref{generalised Legendrian rack of link} that oriented Legendrian links naturally give  rise to $\GL$-racks. The reader may refer to \cite{MR4292403} for more examples of finite type with $\U=\D$. Let us construct another family of examples. Recall that a \textit{permutation rack} $(X, *_\sigma)$ is the rack structure on a set $X$ given by  $x *_\sigma y = \sigma(x)$, where $\sigma$ is a fixed permutation of $X$. Clearly, a permutation rack $(X, *_\sigma)$ is a quandle if and only if $\sigma= \id_X$.
	
	\begin{prop} \label{Leg perm rack}
		Let $X$ be a set and $\sigma$ a permutation of $X$. Then, a pair $(\U,\D)$ of maps is a $\GL$-structure on the permutation rack $(X, *_\sigma)$ if and only if $\D\U=\sigma^{-1}.$
	\end{prop}
	
	\begin{proof}
		The assertion is an immediate consequence of the definition of a $\GL$-rack. 
	\end{proof}
	
	\begin{corollary}\label{cubic permutation rack}
		Let $(X, *_\sigma)$ be a permutation rack. Then, the pairs $(\id_X, \sigma^{-1})$ and  $(\sigma^{-1}, \id_X)$ are $\GL$-structures on $(X, *_\sigma)$. Further, if $\sigma^3 = \id_X$, then the pair $(\sigma,\sigma )$ is also a  $\GL$-structures on $(X, *_\sigma)$.
	\end{corollary}
	
	The preceding corollary provides us a large supply of non-trivial $\GL$-racks. It is a natural question to determine all possible $\GL$-structures on a given rack. In the case when $\D=\U$, we have the following observation for finite permutation racks.
	
	\begin{prop}
		Let $X$ be a finite set with more than one element and $\sigma$ a permutation of $X$. Suppose that there exists an even integer $k$ such that $\sigma$ is a product of odd number of $k$-cycles. Then the permutation rack $(X,*_\sigma)$ does not admit any $\GL$-structure of the form $\U=\D$.
	\end{prop} 	
	
	\begin{proof}
		Note that a permutation $\rho$ of $X$ is a perfect square in the symmetric group on $X$ if and only if, for every even integer $k$, the cycle decomposition of $\rho$ has an even number of $k$-cycles. Hence, the given permutation $\sigma$  cannot be a perfect square. Now, if $(\U, \D)$ is a $\GL$-structure on the permutation rack $(X,*_\sigma)$ with  $\U=\D$, then Proposition \ref{Leg perm rack} gives $\sigma=\U^{-2}$, a contradiction. Hence, $(X,*_\sigma)$ does not admit any $\GL$-structure of the form $\U=\D$.
	\end{proof}		
	
	\begin{remark}
		We note that there is a curious similarity between defining axioms of Legendrian racks considered in \cite[Definition 4]{MR4292403} and quandles with good involutions considered in \cite[Definition 2.1]{MR2657689}. The latter objects have been introduced for coloring unoriented links, and a cohomology theory has also been developed for the same.
	\end{remark}

We conclude this section by recording  some basic categorical aspects of this theory. Let  $\GLcat$ denote the category of $\GL$-racks. Its objects are the $\GL$-racks and the morphisms are as defined above. It is clear that finite direct products  and equalisers exist in $\GLcat$, where the latter are simply the set-theoretic equalisers with the induced $\ast$, $\U_X$ and $\D_X$. Since $\GLcat$ has all binary products and equalisers, they admit all finite limits \cite[Proposition 2.8.2]{MR1291599}.
\par

Let $\GLcat|_X$ denote the slice category over a fixed $\GL$-rack $(X,*,\U_X,\D_X)$. The objects in $\GLcat|_{X}$ are the $\GL$-rack homomorphisms $f:(Y,\star, \U_Y,\D_Y) \rightarrow (X,*,\U_X,\D_X)$, and a morphism from $f:(Y,\star, \U_Y,\D_Y) \rightarrow (X,*,\U_X,\D_X)$ to $g:(Z,\triangleright,\U_Z,\D_Z) \rightarrow (X,*,\U_X,\D_X)$ is a $\GL$-rack homomorphism $h:(Y,\star, \U_Y,\D_Y) \rightarrow (Z,\triangleright,\U_Z,\D_Z)$ such that $g \,h=f$. Since $\GLcat$ has all finite limits, it follows that $\GLcat|_{X}$ also has  finite limits. In particular, the product of $f:(Y,\star, \U_Y,\D_Y) \rightarrow (X,*,\U_X,\D_X)$ and $g:(Z,\triangleright,\U_Z,\D_Z) \rightarrow (X,*,\U_X,\D_X)$ in $\GLcat|_X$ is the pullback $$h:(Y \times_X Z, \bar{\star },\U,\D) \rightarrow (X,*,\U_X,\D_X)$$ of $f$ and $g$, where $Y \times_X Z=\{(y,z) \in Y \times Z \mid f(y)=g(z)\}$ has the product rack structure given by $(y,z) \bar{\star } (y', z')= (y \star y', z\triangleright z')$, $\U(y,z)=(\U_Y(y),\U_Z(z))$, $\D(y,z)=(\D_Y(y),\D_Z(z))$ and $h(y, z)= f(y)=g(z)$ for all $(y,z) \in Y \times_X Z$. Thus, we can consider group objects in $\GLcat|_{X}$.
\medskip

\section{generalised Legendrian rack of a Legendrian link}\label{generalised Legendrian rack of link}
Analogous to the fundamental quandle of an oriented topological link, we define the $\GL$-rack of an oriented Legendrian link in the following manner. We first consider the free $\GL$-rack on a set in the spirit of universal algebra (See \cite[Chapter II, Section 10]{MR0648287}).
	\par
	
	Let $X$ be a non-empty set. We define the \textit{universe of words} generated by $X$ to be the set $W(X)$ satisfying  the following:
	\begin{enumerate}
		\item $x \in W(X)$ for every $x \in X$,
		\item $x\ast y,~ x\ast^{-1} y,~ \U(x),\D(x) \in W(X)$ for every $x,y \in W(X)$.
	\end{enumerate}
	
	\par
	Let $FGLR(X)$ be the set of equivalence classes of elements of $W(X)$ modulo the equivalence relation generated by the following relations:
	\begin{enumerate}
		\item $(x \ast y) \ast^{-1} y \sim x \sim (x \ast^{-1} y) \ast y$,
		\item $(x \ast y) \ast z \sim (x \ast z) \ast (y \ast z)$,
		\item $ \U(\D(x \ast x)) \sim x \sim  \D(\U(x \ast x))$,
		\item $ \U(x \ast y) \sim \U(x) \ast y $,
		\item$ \D(x \ast y)\sim \D(x) \ast y $,
		\item$x  \ast \U(y) \sim x \ast y$,
		\item $x  \ast \D(y) \sim x \ast y$,
	\end{enumerate}
	for all $x,y,z \in W(X)$. In particular, we have maps $\U,\D: FGLR(X) \to FGLR(X)$ given by $x \mapsto \U(x)$ and $x \mapsto \D(x)$. The defining conditions show that  $(FGLR(X), *, \U,\D)$ is a $\GL$-rack, which we refer as the {\it free $\GL$-rack} on $X$.
	\par
	
	\begin{prop}\label{wordsinfreeGLR}
		An element of $FGLR(X)$ has the form $$\U^k\D^l((\cdots((x_1*^{\epsilon_1} x_2)*^{\epsilon_2} x_3) \cdots )*^{\epsilon_{r-1}}x_r)$$ for some $x_i\in X$, $\epsilon_i \in \{1, -1 \}$ and integers $k,l, r$, where $r\ge 1$.
	\end{prop}
	
	\begin{proof}
		It follows from the construction of $FGLR(X)$ and the fact that $\U,\D$ are rack homomorphisms (Proposition \ref{u,d an auto}(1)) that any element of $FGLR(X)$ has a left associated form
		$$
		(\cdots((\U^{n_1}\D^{m_1}(x_1)*^{\epsilon_1} \U^{n_2}\D^{m_2}(x_2))*^{\epsilon_2}  \U^{n_3}\D^{m_3}(x_3) \cdots )*^{\epsilon_{r-1}}\U^{n_r}\D^{m_r}(x_r)).
		$$
		for some $x_i \in X$, $\epsilon_i \in \{1,-1\}$ and integers $n_i,m_i$ and $r$ with $r\ge 1$. Repeated use of axioms (L2), (L2$'$), (L3) and (L3$'$) imply that the preceding expression can be written in the form
		$$
		\U^k\D^l((\cdots((x_1*^{\epsilon_1} x_2)*^{\epsilon_2} x_3) \cdots )*^{\epsilon_{r-1}}x_r)
		$$
		for some integers $k,l$.
	\end{proof}
	
	We now prove the universal property for free $\GL$-racks.
	
	\begin{prop}
		Let $X$ be a non-empty set. Then, for any $\GL$-rack $(L, \star, \U',\D')$ and a set map $\phi :X\rightarrow L$, there exists a unique $\GL$-rack homomorphism $$\tilde{\phi}:(FGLR(X),*, \U,\D)\rightarrow(L, \star, \U',\D')$$ such that the following diagram commutes
		$$
		\begin{tikzcd}
			X  \arrow{d}{\phi} \arrow[rightarrow]{r}{i} & FGLR(X) \arrow[dashed]{ld}{\tilde{\phi}} \\ L.
		\end{tikzcd}
		$$
		Furthermore, the natural map $i$ is injective.
	\end{prop}
	
	\begin{proof}
		Let $\phi:X\rightarrow L$ be any set map. By Proposition \ref{wordsinfreeGLR}, any element $w$ of $FGLR(X)$ has the form 
		$$w=\U^k\D^l((\cdots((x_1*^{\epsilon_1} x_2)*^{\epsilon_2} x_3) \cdots )*^{\epsilon_{r-1}}x_r)$$
		for some $x_i\in X$, $\epsilon_i \in \{1, -1 \}$ and non-negative integers $k, l, r$. We extend the map $\phi$ to $\tilde{\phi}:W(X) \rightarrow L$ by setting 
		$$\tilde{\phi}(w)=(\U')^k(\D')^l ((\cdots((\phi(x_1)\star^{\epsilon_1} \phi(x_2))\star^{\epsilon_2} \phi(x_3)) \cdots )\star^{\epsilon_{r-1}}\phi(x_r))).$$
		For all $a, b, c\in W(X)$, we have
		\begin{small}
			\begin{itemize}
				\item[] $\tilde{\phi}((a\ast b) \ast^{-1} b) = (\tilde{\phi}(a) \star \tilde{\phi}(b)) \star^{-1} \tilde{\phi}(b) = \tilde{\phi}(a)$,
				\item[] $\tilde{\phi}((a \ast b) \ast c) = (\tilde{\phi}(a) \star \tilde{\phi}(b)) \star \tilde{\phi}(c)= (\tilde{\phi}(a) \star \tilde{\phi}(c)) \star (\tilde{\phi}(b) \star \tilde{\phi}(c))=\tilde{\phi}((a *c)*(b *c))$,
				\item[] $\tilde{\phi}(\U(d(a \ast a))=\U'(\D'(\tilde{\phi}(a \ast a))) =\U'(\D'(\tilde{\phi}(a) \star \tilde{\phi}(a))) =\tilde{\phi}(a)$,
				\item[] $\tilde{\phi}(\D(\U(a \ast a))=\D'(\U'(\tilde{\phi}(a \ast a))) =\D'(\U'(\tilde{\phi}(a) \star \tilde{\phi}(a))) =\tilde{\phi}(a)$,
				\item[] $\tilde{\phi}(\U(a \ast b)) =\U'(\tilde{\phi}(a \ast b))=\U'(\tilde{\phi}(a) \star \tilde{\phi}(b))=\U'(\tilde{\phi}(a)) \star \tilde{\phi}(b)= \tilde{\phi}(\U(a))\star  \tilde{\phi}(b)=\tilde{\phi}(\U(a)\ast b)$,
				\item[] $\tilde{\phi}(\D(a \ast b)) =\D'(\tilde{\phi}(a \ast b))=\D'(\tilde{\phi}(a) \star \tilde{\phi}(b))=\D'(\tilde{\phi}(a)) \star \tilde{\phi}(b)= \tilde{\phi}(\D(a))\star  \tilde{\phi}(b)=\tilde{\phi}(\D(a)\ast b)$,
				\item[] $\tilde{\phi}(a \ast \U(b)) =\tilde{\phi}(a) \star \tilde{\phi}(\U(b))=\tilde{\phi}(a) \star \U'(\tilde{\phi}(b))=\tilde{\phi}(a) \star \tilde{\phi}(b)= \tilde{\phi}(a\ast b)$,
				\item[] $\tilde{\phi}(a \ast \D(b)) =\tilde{\phi}(a) \star \tilde{\phi}(\D(b))=\tilde{\phi}(a) \star \D'(\tilde{\phi}(b))=\tilde{\phi}(a) \star \tilde{\phi}(b)= \tilde{\phi}(a\ast b)$.
			\end{itemize}
		\end{small}
		Thus, $\tilde{\phi}$ descends to a rack homomorphism $\tilde{\phi}:FGLR(X) \rightarrow L$ such that $\U' ~\tilde{\phi}=\tilde{\phi}~\U$ and $\D' ~\tilde{\phi}=\tilde{\phi}~\D$. Clearly, $\tilde{\phi}$ is the unique such morphism and the natural map $i$ is injective, which completes the proof.
	\end{proof}
	
	\textbf{Notation.} 
	Let us set some notational conventions for the rest of the paper. 
	\begin{enumerate}
		\item Let $X$ be a non-empty set and $W(X)$ the universe of words generated by $X$ as described above. If $R$ is a given set of relations among elements of $W(X)$, then $\llangle X \mid R \rrangle$ denotes the $\GL$-rack whose elements are equivalence classes of elements of $FGLR(X)$ modulo the equivalence relation generated by $R$.
		\item We write elements of $\GL$-racks as left-associated products. Indeed, it follows from \cite[Lemma 4.4.7]{MR2634013} that any element in a rack can be written in a left-associated form
		\begin{equation*}
			\left(\left(\cdots\left(\left(x_0*^{\epsilon_1}x_1\right)*^{\epsilon_2}x_2\right)*^{\epsilon_3}\cdots\right)*^{\epsilon_{n-1}}x_{n-1}\right)*^{\epsilon_n}x_n,
		\end{equation*}
		which, for simplicity, we write as
		\begin{equation*}
			x_0*^{\epsilon_1}x_1*^{\epsilon_2}\cdots*^{\epsilon_n}x_n.
		\end{equation*}
	\end{enumerate}
	\medskip
	
	We now associate a natural $\GL$-rack to an oriented Legendrian link. Let $K$ be an oriented Legendrian link and $D(K)$ its front diagram with the induced orientation. Let $X= \{x_1,x_2 \ldots ,x_n\}$ be a set with one element for each strand of $D(K)$, where a strand means a connected segment of an arc in the front diagram $D(K)$, which begins and ends at a crossing or a cusp. In other words, a strand does not contain any cusp. At each crossing, we take the usual relation as in the fundamental quandle of the underlying oriented topological link (see Figure \ref{generalised Legendrian rule}(i) and  \ref{generalised Legendrian rule}(ii)), and refer to such a relation as a {\it crossing relation}. The difference arises at cusps. More precisely, if the strand $x$ connects to the strand $y$ at a cusp by moving along the orientation in the positive $z$-direction (see Figure \ref{generalised Legendrian rule}(iii)), then we impose the relation 
	$$y = \U(x).$$
	Similarly, if the strand $x$ connects to the strand $y$ at a cusp by moving along the orientation in the negative $z$-direction (see Figure \ref{generalised Legendrian rule}(iv)), then we impose the relation 
	$$y = \D(x).$$
	We refer to such relations as \textit{cusp relations}. Then, the \textit{$\GL$-rack} $(GLR(D(K)), *, \U,\D)$ associated to the front diagram $D(K)$ is defined to be the set of equivalence classes of elements of $FGLR(X)$ modulo the equivalence relation generated by the crossing and the cusp relations.
	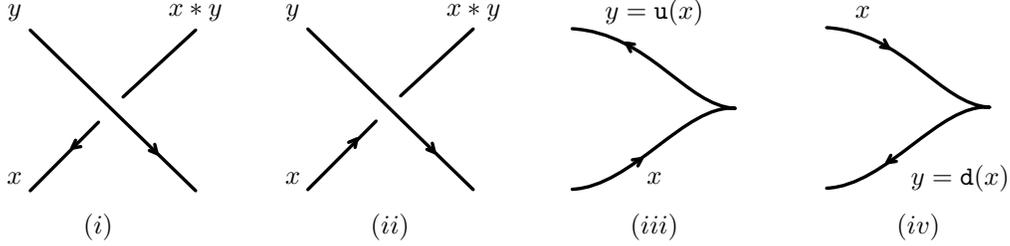
\begin{figure}[h]
		\centering
		\labellist
		\small
		\pinlabel ${x \ast y}$ at 120 130
		\pinlabel {$y$} at -10 130
		\pinlabel {$x$} at -10 10
		\pinlabel {$x$} at 190 10
		\pinlabel {$y$} at 190 130
		\pinlabel {$x \ast y$} at 320 130
		\pinlabel  {$y=\U(x)$} at 450 130
		\pinlabel{$x$} at 450 10
		\pinlabel  {$y=\D(x)$} at 670 10
		\pinlabel{$x$} at 600 130
		\pinlabel {$(i)$} at 50 -25
		\pinlabel {$(ii)$} at 260 -25
		\pinlabel {$(iii)$} at 450 -25
		\pinlabel {$(iv)$} at 640 -25
		\endlabellist
		\includegraphics[width=.8\textwidth]{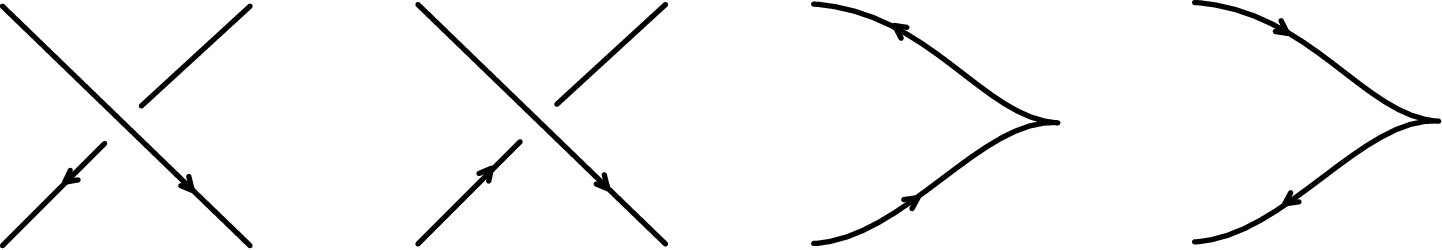}
		\vspace{0.5cm}
		\caption{Relations corresponding to crossings and cusps.}\label{generalised Legendrian rule}
	\end{figure}
	
	\begin{theorem}\label{generalised Legendrian rack invariant}
		The generalised Legendrian rack $(GLR(D(K)), *, \U,\D)$ of a front diagram $D(K)$ of an oriented Legendrian link $K$ is invariant under the Legendrian Reidemeister moves. More precisely, the  generalised Legendrian rack is an invariant of oriented Legendrian links.
	\end{theorem}
	
	\begin{proof}
		We divide the proof into three parts, one for each type of oriented Legendrian Reidemeister move.
		\par
		
		\textbf{LR1 moves.} Consider two front diagrams of $K$ which differ only by the first LR1 move. The LHS of Figure \ref{FirstLRI} contributes four generators $x,y,z,w$ and three relations $y=x \ast w$, $z=\U(y)$ and $w=d(z)$ to presentation of the $\GL$-rack of $K$.  An elementary check shows that these relations are equivalent to the relations $y=x \ast x$, $z=\U(x \ast x)$ and  $w=\D(\U(x \ast x))$. Thus, the generators $y,z,w$ from the LHS can be expressed in terms of the  generator $x$ and the Legendrian structure $(\D, \U)$. At the same time, $x$ is the only generator on the RHS of Figure \ref{FirstLRI}. Hence, the resulting $\GL$-racks of both the front diagrams are isomorphic.
		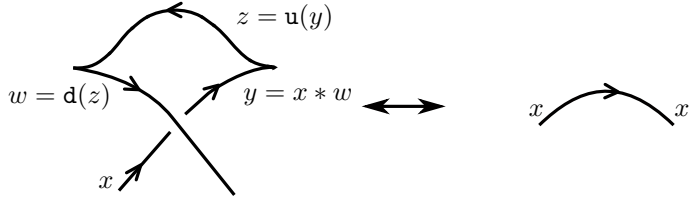
\begin{figure}[h] 
			\centering
			
			\labellist
			\small
			\pinlabel ${z=\U(y)}$ at 160 135
			\pinlabel {$w=\D(z)$} at -10 75
			\pinlabel {$x$} at 25 10
			\pinlabel {$y=x \ast w$} at 170 75
			\pinlabel {$x$} at 460 65
			\pinlabel {$x$} at 350 65
			\endlabellist
			\includegraphics[width=.5\textwidth]{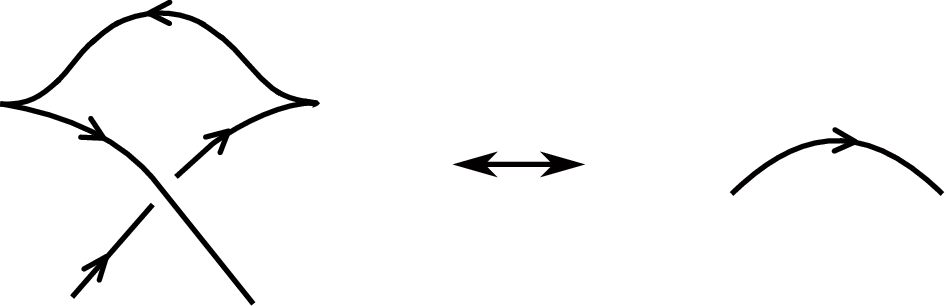}
			\caption{First LR1 move.}\label{FirstLRI}
		\end{figure}
		Similar argument holds for the other LRI moves obtained by changing the orientation and rotating by $180^\circ$ along $x$, $y$ and $z$-axis.
		\par
		\textbf{LR2 moves.} There are 16 second LR2 moves obtained from unoriented Legendrian R2 moves by assigning orientation and rotating $180^{\circ}$ along $x$, $y$ and $z$-axis. 	For the first LR2 move, we see that the LHS of Figure \ref{first LR2} contributes five generators $x,y,a,b,c$ and three relations $a=x \ast y$, $b=\U(x\ast y)$ and $c=b \ast^{-1} y$ to the presentation of the  $\GL$-rack of $K$. These relations are equivalent to the relations $a=x \ast y$, $b=\U(x\ast y)$ and $c=\U(x\ast y) \ast^{-1} y$. That is, the generators $a,b,c$ can be dropped from the presentation of the  $\GL$-rack. The RHS of Figure \ref{first LR2}  contributes three generators $x,y,z$ and one relation $z=\U(x)$. Thus, the  $\GL$-racks obtained from the LHS and RHS are isomorphic.
		\begin{figure}[h] 
			\centering
			\labellist
			\small
			\pinlabel ${c=b \ast^{-1} y}$ at -40 155
			\pinlabel {$y$} at 70 160
			\pinlabel {$x$} at -10 5
			\pinlabel {$b=\U(a)$} at 145 110
			\pinlabel {$a=x \ast y$} at 180 55
			\pinlabel {$z=\U(x)$} at 380 150
			\pinlabel {$x$} at 420 5
			\pinlabel {$y$} at 490 130
			\endlabellist
			\includegraphics[width=.6\textwidth]{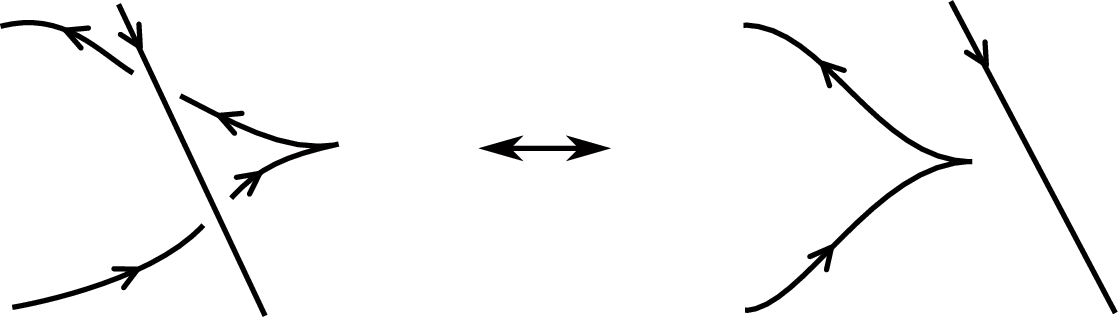}\label{first LR2}
			\caption{First LR2 move.}
			\vspace*{0.5cm}
		\end{figure}
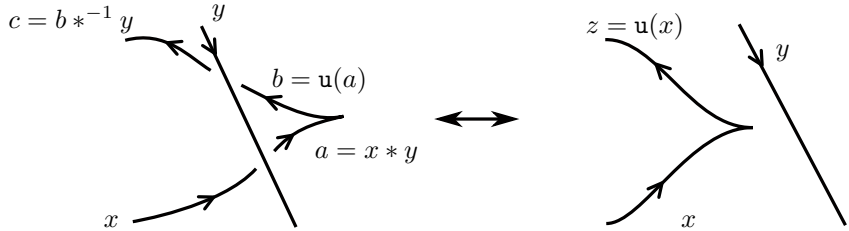
		\par
		
		For the second LR2 move, the LHS of Figure \ref{SecondLRII} gives five generators $x,y,a,b,c$ and three relations  $a=x \ast y$, $b=x \ast y \ast^{-1} \U(x)$ and $c=\U(x)$. Since the RHS of Figure \ref{SecondLRII} gives three generators $x,y,z$ and one relation $z=\U(x)$, it follows that the $\GL$-racks obtained from the LHS and the RHS are isomorphic.
		
		\begin{figure}[h] 
			\centering
			\labellist
			\small
			\pinlabel ${c=\U(x)}$ at -50 140
			\pinlabel {$y$} at 60 5
			\pinlabel {$x$} at -10 5
			\pinlabel {$b=x \ast y \ast^{-1} \U(x)$} at 200 140
			\pinlabel {$a=x \ast y$} at 30 80
			\pinlabel {$z=\U(x)$} at 560 140
			\pinlabel {$x$} at 460 5
			\pinlabel {$y$} at 440 150
			\endlabellist
			\includegraphics[width=.6\textwidth]{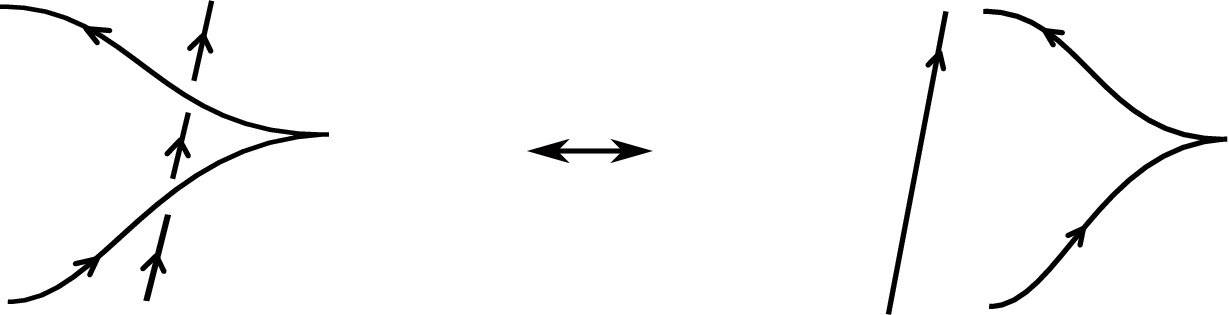}
			\caption{Second LR2 move.}\label{SecondLRII}
		\end{figure}
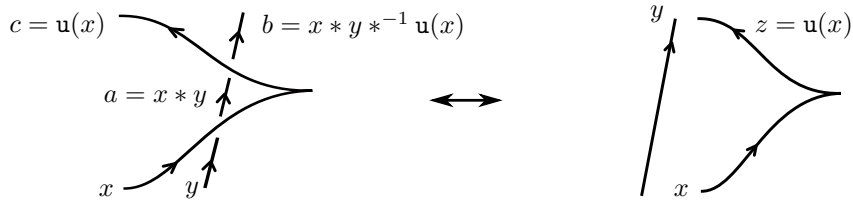
		The other oriented LR2 moves apart from these can be dealt with similarly.
		\par
		\textbf{LR3 moves.} The proof of invariance of  the Legendrian rack under LR3 moves is the same as that of the fundamental quandle of an oriented topological link. This completes the proof of the theorem.
	\end{proof}
	
	In view of Theorem \ref{generalised Legendrian rack invariant}, we have the following definition.
	
	\begin{defn}
		The generalised Legendrian rack $(GLR(K), *, \U, \D)$ of an oriented Legendrian link $K$ is defined to be  $(GLR(D(K)), *, \U, \D)$ for some front diagram $D(K)$ of $K$. 
	\end{defn}
	
	\begin{remark}\label{finite gen of GL rack}
		We note that the $\GL$-rack $(GLR(K), *, \U, \D)$ of an oriented Legendrian link $K$ is finitely generated as a $\GL$-rack since each front diagram has only finitely many crossings and cusps. At the same time, it must be noted that the maps $\U$ and $\D$ are not explicit in general.
	\end{remark}
	
	The following result shows that $\GL$-racks distinguish infinitely many oriented Legendrian unknots.
	
	\begin{theorem}\label{unknots distinguish}
		The Legendrian unknots $U(m,n)$ and $U(k, \ell)$ (as in Figure \ref{unknot}) are not Legendrian isotopic for $(m,n) \neq (k,\ell)$.
	\end{theorem}		
	
	\begin{proof}
		Since $(m,n) \neq (k,\ell)$, either $m \neq k$ or $n \neq \ell$. Without loss of generality, we may assume that $m \neq k$, say, $m < k$. It follows from Figure \ref{unknot} that 
		$$
		(GLR(U(m,n)), \ast, \U_1, \D_1) = \llangle x ~\mid~ \D_1^n\U_1^m(x)=x \rrangle.
		$$
		Similarly, we have
		$$
		(GLR(U(k,\ell)), \ast, \U_2, \D_2) = \llangle y ~\mid~ \D_2^{\ell}\U_2^{k}(y)=y \rrangle.
		$$
		\begin{figure}[h]
			\centering
			\labellist
			\tiny
			\pinlabel ${\U_1^{m-3}(x)}$ at 155 145
			\pinlabel ${\U_1^{m-1}(x)}$ at 155 175
			\pinlabel ${\U_1^{m}(x)}$ at 200 210
			\pinlabel ${\U_1(x)}$ at 220 60
			\pinlabel ${\U_1^{2}(x)}$ at 165 75
			\pinlabel $\D_1^n\U^m(x)=x$ at 220 20
			\pinlabel $\D_1^{n-1}\U^m(x)$ at -20 60
			\pinlabel $\D_1^{n-3}\U^m(x)$ at -20 90
			\pinlabel $\D_1\U^m(x)$ at -10 165
			\pinlabel $\D_1^3\U^{m}(x)$ at -10 135
			\endlabellist
			\includegraphics[width=.3\textwidth]{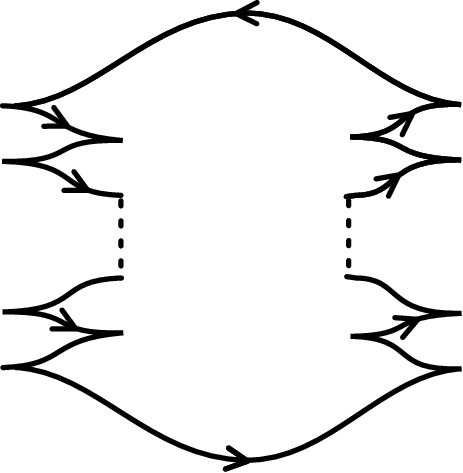}
			\vspace{0.5cm}
			\caption{Unknot of type $U(m,n)$.}\label{unknot}
		\end{figure}
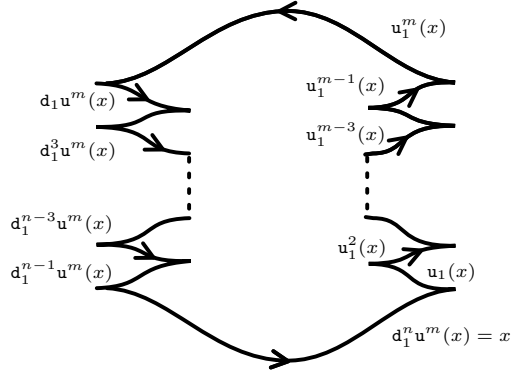	
		Consider the $\GL$ rack $(\mathbb{Z}_k,\ast_{\sigma}, \sigma^{-1}, \id)$, where $\sigma$ is a $k$-cycle on $\mathbb{Z}_k$, the set of integers modulo $k$. Define $\phi:GLR(U(k,\ell)) \rightarrow \mathbb{Z}_k$ by setting $\phi(y)=0$. Then, $\phi$ is a $\GL$-rack homomorphism since $\phi(\D_2^{\ell}\U_2^k(y))=\sigma^{-k}(\phi(y))=\sigma^{-k}(0)=0$. We claim that there does not exist any $\GL$-rack homomorphism from $(GLR(U(m,n)), \ast, \U_1, \D_1)$ to $(\mathbb{Z}_k,\ast_{\sigma}, \sigma^{-1}, \id)$.  Suppose that there exists a $\GL$-rack homomorphism $\psi: GLR(U(m,n)) \rightarrow \mathbb{Z}_k$ given by $\psi(x)=a$ for some $a \in \mathbb{Z}_k$. Then, $\psi(\D_1^n\U_1^m(x))=\sigma^{-m}(\psi(x))=\sigma^{-m}(a)= a$, which is not possible.  Hence, $U(m,n)$ and $U(k,\ell)$ are not Legendrian isotopic.
	\end{proof}

	We now distinguish infinitely many oriented Legendrian trefoils.
	
	\begin{theorem}\label{trefoils distinguish}
		The Legendrian trefoils $T(m,n)$ and $T(k,\ell)$ (as in Figure \ref{trefoil}) are not Legendrian isotopic for $(m,n)\neq (k, \ell)$.
	\end{theorem}
	
	\begin{proof}
		Using cusp relations, crossing relations and axioms of a $\GL$-rack, one can conclude that the presentation of the Legendrian trefoil $T(m,n)$ as shown in Figure \ref{trefoil} is given by
		$$ (GLR(T(m,n)), \ast, \U_1, \D_1) = \llangle a,g ~\mid~ g \ast a = \U_1^{2m}\D_1^{2n}(a \ast^{-1} g), \quad a \ast g= \D_1 \U_1^{-1}(g \ast^{-1}a) \rrangle. $$
		Similarly, we have
		$$ (GLR(T(k,\ell)), \ast, \U_2, \D_2) = \llangle a',g' ~\mid~ g' \ast a' = \U_2^{2k}\D_2^{2\ell}(a' \ast^{-1} g'), \quad a' \ast g'= \D_2 \U_2^{-1}(g' \ast^{-1}a') \rrangle. $$
		Since $(m,n) \neq (k,\ell)$,  we may assume that $m < k$. Consider the $\GL$-rack $(\mathbb{Z}_{2k+3}, *_\sigma, \sigma^{-1}, \id)$, where $\sigma(x)=x+1 \mod (2k+3)$. Define $\phi:GLR(T(k,\ell)) \rightarrow \mathbb{Z}_{2k+3}$ by setting $\phi(a')=0$ and $\phi(g')=1$. Then $\phi(g' \ast a')=2$ and $\phi(\U_2^{2k} \D_2^{2 \ell}(a' \ast^{-1} g'))= \sigma^{-2k}(2k+2) = 2$. Similarly, $\phi(a' \ast g') =1$ and $ \phi(\U_2^{-1} \D_2(g' \ast^{-1}a'))=\sigma (\phi (g' \ast^{-1}a'))=\sigma(0)=1$. Thus, $\phi$ is a $\GL$-rack homomorphism.  For the other trefoil, suppose that there exists a $\GL$-rack homomorphism $\psi:GLR(T(m,n)) \rightarrow \mathbb{Z}_{2k+3}$ defined by $\psi(a)=x$ and $\psi(g)=y$ for some $x,y \in \mathbb{Z}_{2k+3}$. Applying $\psi$ to the relations gives $\sigma^{2m+3}(x)=x$, which is not possible as $m<k$. Hence, $T(m,n)$ and $T(k,\ell)$ are not Legendrian isotopic for $(m,n) \neq (k, \ell)$.
		
		\begin{figure}[h]
			\centering
			\labellist
			\small
			\pinlabel $b$ at 120 330
			\pinlabel $a$ at 90 30
			\pinlabel $f$ at 390 330
			\pinlabel $x_1$ at 435 73
			\pinlabel $x_2$ at 430 125
			\pinlabel $i$ at 270 80
			\pinlabel $e$ at 320 20
			\pinlabel $d$ at 170 120
			\pinlabel $h$ at 170 167
			\pinlabel $g$ at 170 250
			\pinlabel $y_1$ at 247 279
			\pinlabel $y_2$ at 290 255
			\pinlabel $y_{2n-2}$ at 300 190
			\pinlabel $c$ at 270 155
			\pinlabel $x_{2m-2}$ at 450 269
			\endlabellist
			\includegraphics[width=.4\textwidth]{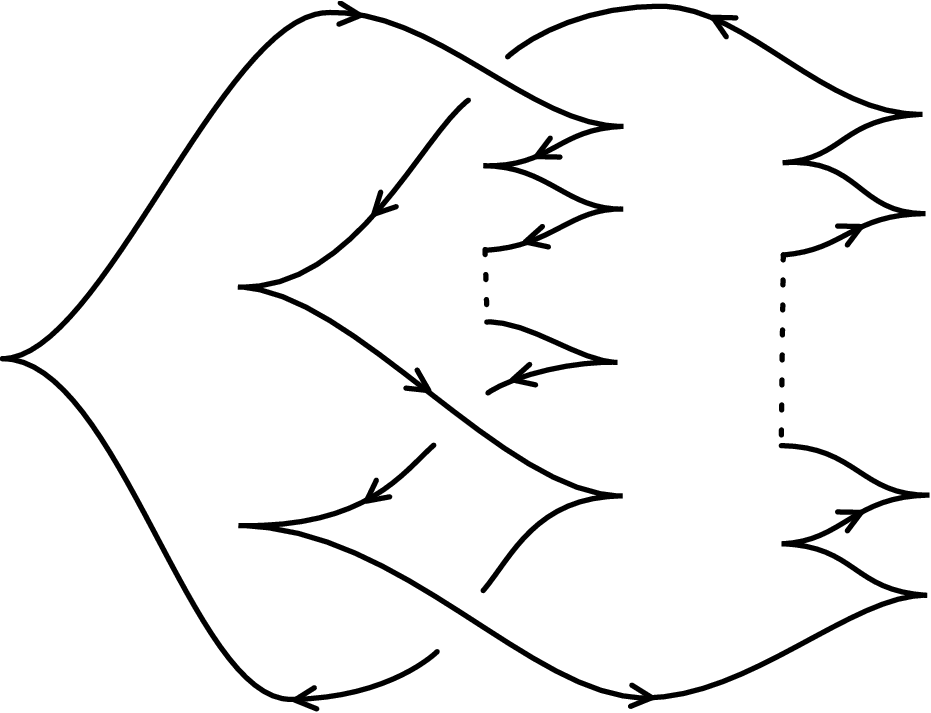}
			\vspace{0.5cm}
			\caption{Trefoil of type $T(m,n)$.}\label{trefoil}
		\end{figure}
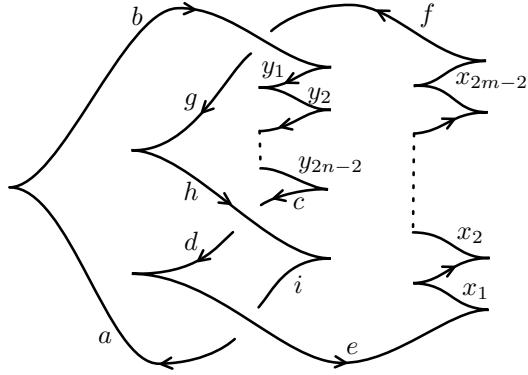
	\end{proof}
	
	\begin{remark}
		Theorems \ref{unknots distinguish} and \ref{trefoils distinguish} provide us a good deal of examples with the following properties.
		\begin{enumerate}
			\item Rotation number does not differentiate the following pairs of knots  but $\GL$-rack does:
			\begin{enumerate}
				\item  Unknots of the type $U(m,m)$ and $U(n,n)$ for $m \neq n$.
				\item Trefoils of the type $T(k,k)$ and $T(\ell,\ell)$ for $k \neq \ell$. 
			\end{enumerate}	
			\item Thurston-Bennequin number does not differentiate  the following pairs of knots but $\GL$-rack does:
			\begin{enumerate}
				\item Unknots of the type $U(1,2m-1)$ and $U(m,m)$ for odd $m  \geq 3$.
				\item Trefoils of the type $T(1,2n-1)$ and $T(n,n)$ for odd $n  \geq 3$.
			\end{enumerate}
		\end{enumerate}
	\end{remark}
	
	\begin{remark}
		Note that the oriented Legendrian unknots as in Figure \ref{unknots} have isomorphic $\GL$-racks. Further, it is known due to Eliashberg and Fraser \cite[Theorem 1.5]{MR2496415} that these two unknots are Legendrian isotopic.
		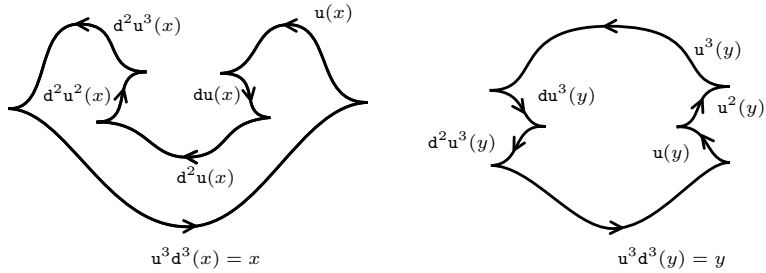
\begin{figure}[h]
			\centering
			\labellist
			\tiny
			\pinlabel $\U^3\D^3(x)=x$ at 170 -20
			\pinlabel $\D \U(x)$ at 180 120
			\pinlabel $\D^2\U(x)$ at 170 50
			\pinlabel $\D^2\U^2(x)$ at 60 120
			\pinlabel $\D^2\U^3(x)$ at 120 180
			\pinlabel $\U(x)$ at 280 190
			\pinlabel $\U^3\D^3(y)=y$ at 570 -20
			\pinlabel $\U(y)$ at 570 70
			\pinlabel $\U^2(y)$ at 630 110
			\pinlabel $\U^3(y)$ at 610 160
			\pinlabel $\D^2\U^3(y)$ at 390 80
			\pinlabel $\D\U^3(y)$ at 480 120
			\endlabellist
			\includegraphics[width=.6\textwidth]{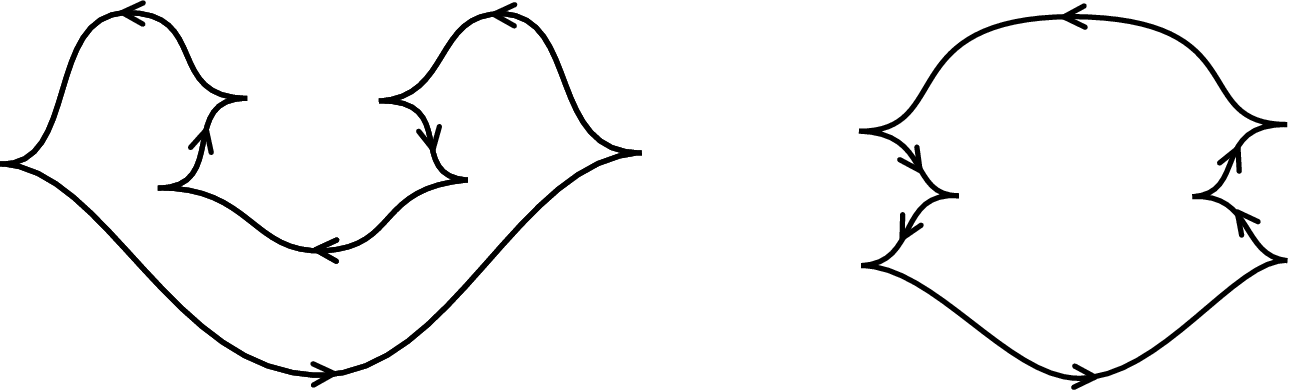}
			\vspace{0.5cm}
			\caption{Legendrian unknots.}\label{unknots}
		\end{figure}
	\end{remark}	
	
	\begin{remark}
		Consider the oriented Chekanov knots as in Figure \ref{chekanovknots}. Then their $\GL$-racks are given by
		\begin{eqnarray*}
			GLR(K_1)=\llangle e_1,m_1 &\mid& \D_1^3\U_1^2((e_1 \ast m_1)\ast e_1) \ast ((m_1 \ast e_1) \ast m_1)=m_1,\\
			& & \D_1^2\U_1^3((m_1 \ast e_1)\ast m_1) \ast ((e_1 \ast m_1) \ast e_1)=e_1 \rrangle
		\end{eqnarray*}
		and 
		\begin{eqnarray*}
			GLR(K_2)=\llangle e_2,m_2 &\mid& \D_2^3\U_2^4((e_2 \ast m_2)\ast e_2) \ast ((m_2 \ast e_2) \ast m_2)=m_2, \\
			& &\D_2^2\U_2((m_2 \ast e_2)\ast m_2) \ast ((e_2 \ast m_2) \ast e_2)=e_2 \rrangle.
		\end{eqnarray*}
		Define $\psi: GLR(K_1) \rightarrow GLR(K_2)$ by $\psi(e_1)=\U_2(e_2)$ and $\psi(m_1)=\U_2^{-1}(m_2)$, and $\phi: GLR(K_2) \rightarrow GLR(K_1)$ by $\phi(e_2)=\U_1^{-1}(e_1)$ and $\psi(m_2)=\U_1(m_1)$. It follows that both $\phi$ and $\psi$ can be extended to $GL$-rack homomorphisms such that $\phi ~ \psi=\id$ and $\psi ~ \phi = \id$. However, these two knot are famously known to be Legendrian non-isotopic due to Chekanov \cite[Theorem 5.8]{MR1946550}. Thus, the $GL$-rack is not a complete invariant.	
		\begin{figure}[h]
			\centering
			\labellist
			\tiny	
			\pinlabel $b_1$ at 200 130
			\pinlabel $g_1$ at 50 167
			\pinlabel $m_1$ at 0 167
			\pinlabel $e_1$ at 425 330
			\pinlabel $h_1$ at 270 80
			\pinlabel $n_1$ at 320 -20
			\pinlabel $p_1$ at 300 430
			\pinlabel $f_1$ at 300 490
			\pinlabel $a_1$ at 140 167
			\pinlabel $q_1$ at 200 290
			\pinlabel $j_1$ at 260 350
			\pinlabel $i_1$ at 300 320
			\pinlabel $k_1$ at 282 270
			\pinlabel $c_1$ at 260 190
			\pinlabel $\ell_1$ at 270 230
			\pinlabel $o_1$ at 380 330
			\pinlabel $\ell_2$ at 810 130
			\pinlabel $g_2$ at 670 167
			\pinlabel $m_2$ at 600 167
			\pinlabel $h_2$ at 870 80
			\pinlabel $n_2$ at 920 -20
			\pinlabel $p_2$ at 900 440
			\pinlabel $f_2$ at 900 500
			\pinlabel $a_2$ at 690 230
			\pinlabel $q_2$ at 792 167
			\pinlabel $j_2$ at 870 330
			\pinlabel $i_2$ at 930 300
			\pinlabel $k_2$ at 880 270
			\pinlabel $c_2$ at 770 250
			\pinlabel $b_2$ at 820 200
			\pinlabel $o_2$ at 985 330
			\pinlabel $e_2$ at 1040 330
			\endlabellist
			\includegraphics[width=.8\textwidth]{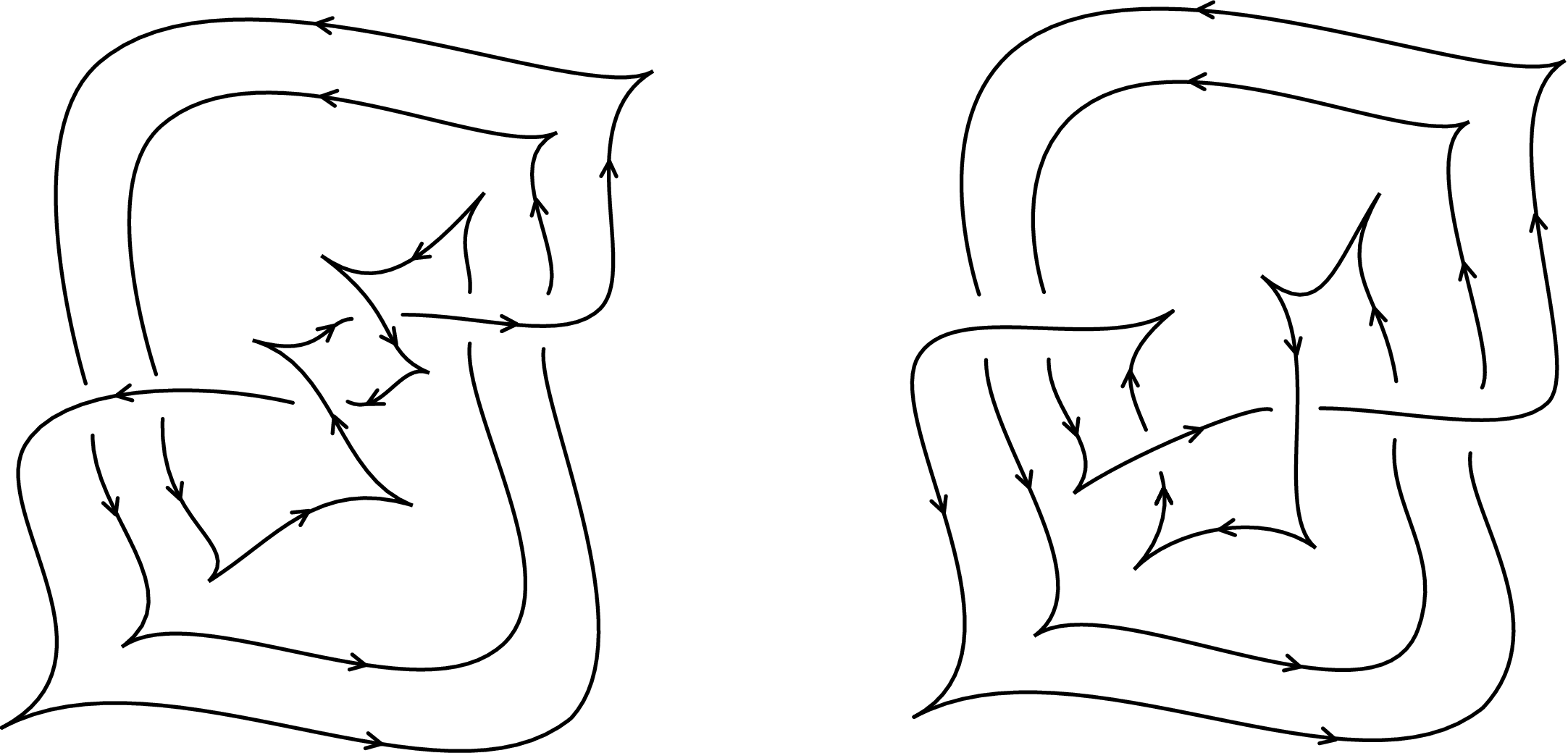}
			\vspace{0.5cm}
			\caption{Chekanov knots $K_1$ and $K_2$.}\label{chekanovknots}
		\end{figure}
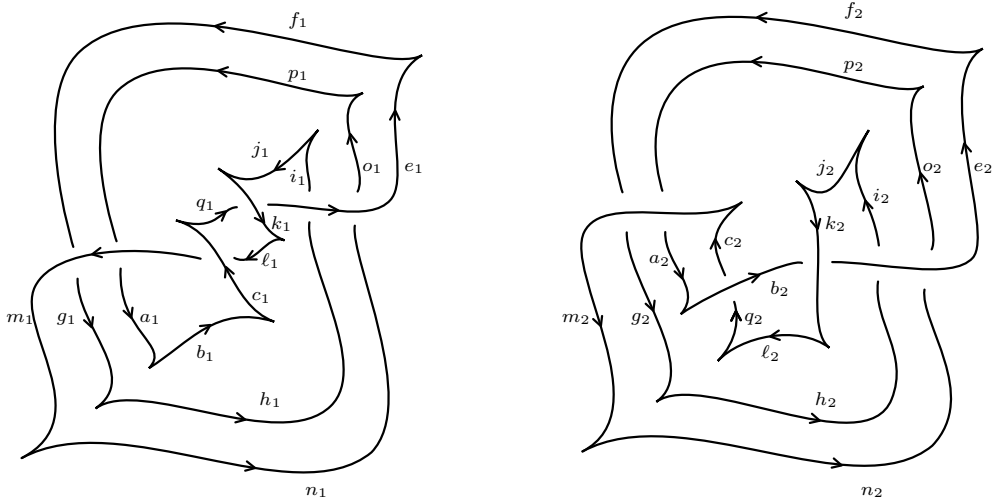
	\end{remark}

	\begin{theorem} \label{legendrian link distinguishes distinct top links}
		If $K_1$ and $K_2$ are two topologically distinct non-split oriented Legendrian links, then $(GLR(K_1), *, \U,\D) \not\cong (GLR (K_2), *, \U',\D')$.
	\end{theorem}
	
	\begin{proof}
		Note that the fundamental quandle of the underlying oriented topological link can be obtained from the $\GL$-rack of the oriented Legendrian link by adding the relation  $\U(x)=\D(x)=x$ for each generators $x$ of the  $\GL$-rack. Further, recall that the fundamental quandle is a complete invariant of a non-split oriented link upto orientation and mirror-isotopy. Since $K_1$ and $K_2$ are topologically distinct  non-split oriented Legendrian links, it follows that $Q(K_1) \not\cong Q(K_2)$, and hence $(GLR(K_1), *, \U,\D) \not\cong (GLR (K_2), *, \U',\D')$.
	\end{proof}

Motivated by theorems \ref{unknots distinguish} and \ref{trefoils distinguish}, we define the notion of colouring of an oriented Legendrian link by a  $\GL$-rack.

\begin{defn}
Let $K$ be an oriented Legendrian link,  $D(K)$ its front diagram and $S$ the set of strands in $D(K)$. A coloring of  $D(K)$ by a given $\GL$-rack $(X, \star, \U',\D')$ is a map $\mathcal{C}: S \to X$ such that 
	\begin{enumerate}
		\item  $\mathcal{C}(x) \star \mathcal{C}(y)=\mathcal{C}(z)$,
		\item $ \mathcal{C}(b)= \U'(\mathcal{C}(a))$,
		\item $ \mathcal{C}(a)= \D'(\mathcal{C}(b))$,
	\end{enumerate}
	where $x, y, z, a, b$ are the strands as shown in Figure \ref{generalised Legendrian colouring}.
	\begin{figure}[h] 
		\centering
		\labellist
		\small
		\pinlabel ${z}$ at 120 130
		\pinlabel {$y$} at -5 130
		\pinlabel {$x$} at -5 10
		\pinlabel {$x$} at 190 10
		\pinlabel {$y$} at 190 130
		\pinlabel {$z$} at 320 130
		\pinlabel  {$b=\U(a)$} at 460 130
		\pinlabel{$a$} at 460 10
		\pinlabel  {$a=\D(b)$} at 680 10
		\pinlabel{$b$} at 600 130
		\pinlabel{$\mathcal{C}(x) \star \mathcal{C}(y)=\mathcal{C}(z)$} at 180 -30
		\pinlabel{$\mathcal{C}(b)=\U'(\mathcal{C}(a))$} at 440 -30
		\pinlabel{$\mathcal{C}(a)=\D'(\mathcal{C}(b))$} at 640 -30
		\endlabellist
		\includegraphics[width=.7\textwidth]{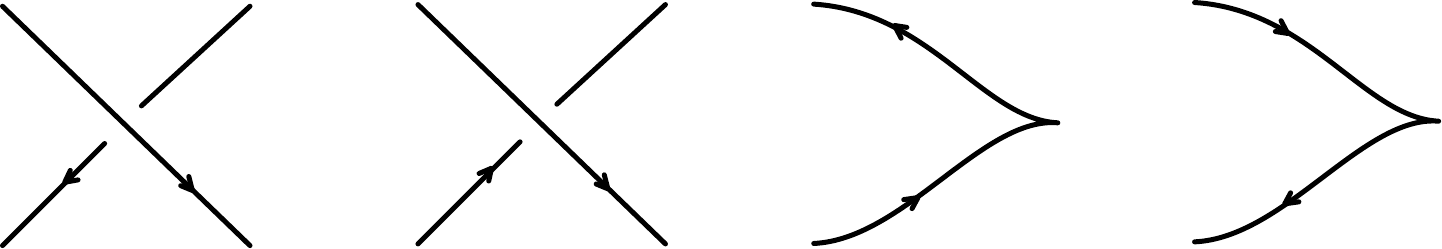}
		\vspace*{0.5cm}
		\caption{Coloring conditions at crossings and cusps.}\label{generalised Legendrian colouring}
	\end{figure}
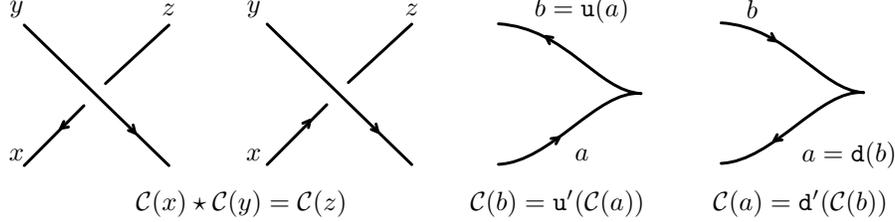
\end{defn}

It is not difficult to see that the total number of colorings of $D(K)$  by  $(X, \star,\U',\D')$ is an invariant of the link $L$, that is, it does not depend on the front diagram of $K$. Furthermore, it follows from the construction of $(GLR(K), *,\U,\D)$  that the number of colorings equals the cardinality of the set $\Hom_{GLR}(K, X)$ of all  $\GL$-rack homomorphisms from $(GLR(K), *,\U,\D)$ to $(X, \star,\U',\D')$. In view of Theorem \ref{generalised Legendrian rack invariant}, we have the following result.

\begin{corollary}\label{cor generalised Legendrian coloring}
If $K$ is an oriented Legendrian link and $(X, \star,\U',\D')$ a  $\GL$-rack, then $\Hom_{GLR}(K, X)$ is an invariant of $K$.
\end{corollary}

\begin{remark}\label{finiteness GL morphisms}
It follows from the definition of a morphism between two $\GL$-racks and Proposition \ref{wordsinfreeGLR} that a morphism from  $(GLR(K), *,\U,\D)$  to  $(X, \star,\U',\D')$  is completely determined by its values on the set of strands in a front diagram $D(K)$ of $K$. Since the set of strands is always finite, the set $\Hom_{GLR}(K, X)$ is also finite whenever $X$ is finite.
\end{remark}
	\medskip

	\section{Homogeneous representation of a generalised Legendrian rack}\label{homogenous representation}
	This section is motivated by \cite{Joyce1979}, where a similar result is proved for quandles.  Let $G$ be a group, $I$ an indexing set, $\{z_i \mid i \in I \}$ a subset of $G$ and $\{H_i \mid i \in I \}$ a family of subgroups of $G$ such that  $H_i \le \C_G(z_i)$ for all $i \in I$. Then, the disjoint union $\sqcup_{i \in I} ~G/H_i$ of left cosets forms a rack, denoted $(\sqcup_{i \in I}~G/H_i,~\{z_i \}_{i \in I})$, by setting $$xH_i \ast yH_j = yz_jy^{-1}xH_i$$ for $xH_i,yH_j \in \sqcup_{i \in I}~G/H_i$.

	\begin{prop}\label{legendrian rack on coset}
		Let $G$ be a group, $I$ an indexing set, $\{z_i \mid i \in I \}$ a subset of $G$ and $\{H_i \mid i \in I \}$ a family of subgroups of $G$ such that  $H_i \le \C_G(z_i)$ for all $i \in I$.  Further, suppose that there are subsets $\{ r_i \mid i \in I \}$ and $\{s_i \mid i \in I \}$ of $G$, and a bijection $\tau: I \to I$ with  $\mu:I \to I$ as its inverse such that the following conditions are satisfied:
		\begin{enumerate}
			\item $r_i^{-1} h_i r_i \in H_{\tau(i)}$ for all $i \in I$ and $h_i \in H_i$.
			\item $s_i^{-1} h_i s_i \in H_{\mu(i)}$ for all $i \in I$ and $h_i \in H_i$.
			\item $z_i r_i s_{\tau(i)} \in H_i$ for all $i \in I$.
			\item $z_i s_i r_{\mu(i)} \in H_i$ for all $i \in I$.
			\item $z_i r_i=r_i z_{\tau(i)}$ for all $i \in I$.
			\item $z_i s_i=s_i z_{\mu(i)}$ for all $i \in I$.
		\end{enumerate}
		Then, the maps
		$$\mathtt{u}_r,~ \mathtt{d}_s:\sqcup_{i \in I} ~G/H_i \longrightarrow \sqcup_{i \in I} ~G/H_i$$
		defined by $\mathtt{u}_r(xH_j )= xr_jH_{\tau(j)}$ and $\mathtt{d}_s(xH_j)= xs_j H_{\mu(j)}$ form a Legendrian structure on the rack $(\sqcup_{i \in I}~G/H_i,~\{z_i \}_{i \in I})$.
	\end{prop}
	
	\begin{proof}
		Suppose that $xH_i=yH_i$, that is, there exists $h \in H_i$ such that $x=yh$. Then 
		$$\mathtt{u}_r(xH_i)=xr_iH_{\tau(i)}=yhr_iH_{\tau(i)}=yr_i(r_i^{-1}hr_i)H_{\tau(i)}=yr_iH_{\tau(i)}=\mathtt{u}_r(yH_i)$$
		and
		$$\mathtt{d}_s(xH_i)=xs_iH_{\mu(i)}=yhs_iH_{\mu(i)}=ys_i(s_i^{-1}hs_i)H_{\mu(i)}=ys_iH_{\mu(i)}=\mathtt{d}_s(yH_i).$$
		Thus, the maps $\mathtt{u}_r$ and $\mathtt{d}_s$ are well-defined. For $xH_i,yH_j \in \sqcup_{i \in I}~G/H_i$, a direct computation gives
		$$\mathtt{u}_r \mathtt{d}_s(xH_i \ast xH_i)=\mathtt{u}_r \mathtt{d}_s (xz_iH_i)=xz_i s_i r_{\mu(i)}H_{\tau \mu(i)}=xH_i,$$
		$$\mathtt{d}_s \mathtt{u}_r(xH_i \ast xH_i)=\mathtt{d}_s \mathtt{u}_r (xz_iH_i)=xz_i r_i s_{\tau(i)}H_{\mu \tau(i)}=xH_i,$$
		$$\mathtt{u}_r(xH_i \ast yH_j)= \mathtt{u}_r(yz_jy^{-1}xH_i)=yz_jy^{-1}x r_i H_{\tau(i)}= xr_iH_{\tau(i)} \ast yH_j= \mathtt{u}_r(xH_i)\ast yH_j,$$
		$$ \mathtt{d}_s(xH_i \ast yH_j)= \mathtt{d}_s(yz_jy^{-1}xH_i)=yz_jy^{-1}x s_i H_{\mu(i)}= xs_iH_{\mu(i)} \ast yH_j= \mathtt{d}_s(xH_i)\ast yH_j,$$
		$$xH_i \ast \mathtt{u}_r(yH_j)= xH_i \ast yr_jH_{\tau(j)} = yr_j z_{\tau(j)} r_j^{-1}y^{-1} xH_i=yz_jy^{-1}xH_i= xH_i \ast yH_j$$
		and
		$$xH_i \ast \mathtt{d}_s(yH_j)= xH_i \ast y s_j H_{\mu(j)} = ys_j z_{\mu(j)} s_j^{-1}y^{-1} xH_i=yz_jy^{-1}xH_i= xH_i \ast yH_j.$$
		Hence, $\mathtt{u}_r$ and $\mathtt{d}_s$ form a Legendrian structure on $(\sqcup_{i \in I}~G/H_i,~\{z_i \}_{i \in I})$.
	\end{proof}
	
	We denote the $\GL$-rack obtained in the preceding proposition by \\
	$(\sqcup_{i \in I}~G/H_i,~\{z_i \}_{i \in I},~\{r_i \}_{i \in I}, ~~\{s_i \}_{i \in I})$. The next result is a Legendrian rack analogue of \cite[Proposition, Section 2.4]{Joyce1979}.
	
	\begin{theorem}\label{thm homogenous representation}
		Every $\GL$-rack is isomorphic to a $\GL$-rack of the form \\
		$(\sqcup_{i \in I}~G/H_i,~\{z_i \}_{i \in I},~\{r_i \}_{i \in I}, ~~\{s_i \}_{i \in I})$.
	\end{theorem}
	
	\begin{proof}
		Let $(X, \ast, \mathtt{u}, \mathtt{d})$ be a $\GL$-rack and $G=\Aut(X, \ast, \mathtt{u}, \mathtt{d})$ its group of Legendrian automorphisms. Let $X= \sqcup_{i \in I}~X_i$ be the decomposition of $X$ into $G$-orbits. Noting that the Legendrian map $\mathtt{u}$ permutes the orbits of $X$, we obtain a map $\tau: I \to I$ such that $i \mapsto \tau(i)$ if and only if $\mathtt{u}(X_i) \subseteq X_{\tau(i)}$. Similarly, since $\mathtt{d}$ permutes the orbits of $X$, we obtain a map $\mu: I \to I$ such that $i \mapsto \mu(i)$ if and only if $\mathtt{d}(X_i) \subseteq X_{\mu(i)}$. If $x \in X$, then axiom (L2) gives $\mathtt{d}(x*x)=\mathtt{d}(x)*x$. Applying $\mathtt{u}$ and using axiom (L1) gives $$x=\mathtt{u}\mathtt{d}(x*x)=\mathtt{u}(\mathtt{d}(x)*x)=\mathtt{u}\mathtt{d}(x)*x=S_x(\mathtt{u}\mathtt{d}(x)).$$ Since $S_x \in G$, it follows that $x$ and $\mathtt{u}\mathtt{d}(x)$ are in the same orbit, that is, the composition $\mathtt{u}\mathtt{d}$ preserves orbits of $X$. This implies that $\tau \mu=\id_I= \mu \tau$. For each $i \in I$, fix elements $p_i \in X_i$ and $\phi_i, \psi_i \in G$ such that 
		$$\phi_i(p_{\tau(i)})=\mathtt{u}(p_i) \quad \textrm{and} \quad \psi_i(p_{\mu(i)})=\mathtt{d}(p_i).$$
		We set $r_i=\phi_i$, $s_i=\psi_i$, $z_i = S_{p_i}$ and $H_i$ the stabiliser of $p_i$ in $G$. Then, for each $i \in I$,  $\xi \in H_i$ and  $x \in X$, we have
		$$S_{p_i} \xi (x)= \xi (x)* p_i=  \xi (x)* \xi(p_i)= \xi(x*p_i)= \xi S_{p_i}(x),$$ and hence we obtain the rack $(\sqcup_{i \in I}~G/H_i,~\{z_i \}_{i \in I})$.
		For any $\xi \in H_i$, we have
		$$\phi_i^{-1} \xi \phi_i(p_{\tau(i)})=\phi_i^{-1} \xi \mathtt{u}(p_i)= \phi_i^{-1} \mathtt{u} \xi(p_i)= \phi_i^{-1} \mathtt{u} (p_i)=p_{\tau(i)},$$
		and hence $\phi_i^{-1} \xi \phi_i \in H_{\tau(i)}$. Similarly, one can show that $\psi_i^{-1} \xi \psi_i \in H_{\mu(i)}$. It follows from axiom (L1) that
		$$ S_{p_i}\phi_i \psi_{\tau(i)}(p_i)=S_{p_i}\phi_i \mathtt{d}(p_{\tau(i)})=S_{p_i} \mathtt{d} \phi_i (p_{\tau(i)})= S_{p_i} \mathtt{d} \mathtt{u} (p_i)= \mathtt{d}\mathtt{u} (p_i)*p_i=p_i,$$
		and hence $S_{p_i}\phi_i \psi_{\tau(i)}  \in H_i$. Similarly, it follows that $S_{p_i}\psi_i \phi_{\mu(i)}  \in H_i$. Also, axiom (L3)  implies that 
		$$\phi_i S_{p_{\tau(i)}}(x)= \phi_i (x*p_{\tau(i)})=\phi_i (x)*\phi_i(p_{\tau(i)})=\phi_i (x)*\mathtt{u}(p_i)= \phi_i (x)*p_i= S_{p_i} \phi_i(x)$$
		and  $\psi_i S_{p_{\mu(i)}}(x)= S_{p_i} \psi_i(x)$ for all $x \in X$. Hence,  Proposition \ref{legendrian rack on coset} yields the Legendrian rack $(\sqcup_{i \in I}~G/H_i,~\{z_i \}_{i \in I},~\{r_i \}_{i \in I},~ \{s_i \}_{i \in I})$. It can be seen that the orbit map 
		$$\Theta:\sqcup_{i \in I}~G/H_i \rightarrow X$$ given by $\Theta(\xi H_j)= \xi (p_j)$ is an isomorphism of racks (see \cite[Section 2.4]{Joyce1979}). Furthermore, for any $\xi H_j \in \sqcup_{i \in I}~G/H_i$, we have
		$$\Theta \mathtt{u}_r(\xi H_j)=\Theta(\xi\phi_j H_{\tau(j)})=\xi\phi_j(p_{\tau(j)})= \xi \mathtt{u}(p_j)=\mathtt{u} \xi(p_j)= \mathtt{u} \Theta (\xi H_j)$$
		and 	$\Theta \mathtt{d}_s(\xi H_j)= \mathtt{d} \Theta (\xi H_j)$. Thus, $\Theta$ is an isomorphism of $\GL$-racks.
	\end{proof}
	\medskip


\section{Modules over $\GL$-racks} \label{sec GL modules}
This section is devoted to identifying appropriate coefficient objects for constructing (co)homology theories for $\GL$-racks. Introduced by Beck in \cite{MR2616383}, these coefficient objects are commonly known as \textit{Beck modules} in the literature. More precisely, given a category $\mathbf{C}$ and an object $X$ in $\mathbf{C}$, a Beck module over $X$ is an abelian group object in the slice category $\mathbf{C}/X$ over $X$. To exemplify their strength, we note that they simultaneously generalise the notion of coefficient modules from group cohomology, Lie algebra cohomology, as well as Hochschild cohomology of associative algebras. To avoid confusion, from now onwards, we shall write the maps $\U$ and $\D$ of a $\GL$-rack $(X, \ast,\U,\D)$ by $\U_X$ and $\D_X$, respectively.
	\par
We begin by recalling some basic concepts from the theory of trunks, which was introduced by Fenn, Rourke, and Sanderson in \cite{MR1364012}. A \textit{trunk} $\T$ is an object analogous to a category which consists of a class of objects and a set $\Hom_{\T}(A,B)$ of morphisms along with a number of commutative squares
	\begin{center}
		\begin{tikzcd}[column sep=3em]
			A \arrow[r, "f"] \arrow[d, "g"] & B \arrow[d, "h"] \\
			C \arrow[r, "i"] & D,
		\end{tikzcd}
	\end{center}
	called {\it preferred squares}. Given two trunks $\Sa$ and $\T$, a \textit{trunk map} $F:\Sa \rightarrow \T$ is a map which assigns to each object $A$ of $\Sa$, an object $F(A)$ of $\T$, and to every morphism $f:A \rightarrow B$ of $\Sa$, a morphism $F(f):F(A) \rightarrow F(B)$ of $\T$ such that the preferred square
	\begin{center}
		\begin{tikzcd}[column sep=3em]
			A \arrow[r, "f"] \arrow[d, "g"] & B \arrow[d, "h"] \\
			C \arrow[r, "i"] & D
		\end{tikzcd}
	\end{center}
	is mapped to the preferred square
	\begin{center}
		\begin{tikzcd}[column sep=3em]
			F(A) \arrow[r, "F(f)"] \arrow[d, "F(g)"] & F(B) \arrow[d, "F(h)"] \\
			F(C) \arrow[r, "F(i)"] & F(D).
		\end{tikzcd}
	\end{center}
	
	For example, for any category $\C$, we have a well-defined trunk $\T(\C)$, which has the same objects and morphisms as $\C$, and whose preferred squares are the commutative diagrams in $\C$. In particular, we denote the trunk associated to the category of abelian groups by $\Ab$.
	
	Given a $\GL$-rack $(X, \ast,\U_X,\D_X)$, we define a trunk $\T(X, \ast,\U_X,\D_X)$ as follows: 
	\begin{itemize}
		\item  for each $x \in X$, $\T(X, \ast,\U_X,\D_X)$ has precisely one object,
		\item for each $x \in X$, there are morphisms $\mu_x:x \rightarrow \U_X(x)$, $\partial_x:x \rightarrow \D_X(x)$ and $\id_x: x \rightarrow x$,
		\item for each ordered pair $(x,y)$ of elements of $X$, there are morphisms $\alpha_{x,y}:x \rightarrow x \ast y$ and $\beta_{x,y}:y \rightarrow x *y$ such that the following are preferred squares for all $x,y,z \in X$:
		\begin{center}
			\hspace{-0.3cm}
			\begin{tikzcd}[row sep=3em, column sep=5em]
				x \arrow[r, "\alpha_{x,y}", shorten >=2pt, shorten <=2pt] \arrow[d, "\alpha_{x,z}"', shorten >=0pt, shorten <=0pt] & x \ast y \arrow[d, "\alpha_{x \ast y,z}"', shorten >=0pt, shorten <=0pt] \\
				x \ast z \arrow[r, "\alpha_{x\ast z, y \ast z}"'] & (x\ast y) \ast z
			\end{tikzcd}
			\hspace{2cm}
			\begin{tikzcd}[row sep=3em, column sep=5em]
				y \arrow[r, "\beta_{x,y}", shorten >=2pt, shorten <=2pt] \arrow[d, "\alpha_{y,z}"', shorten >=0pt, shorten <=0pt] & x \ast y \arrow[d, "\alpha_{x \ast y,z}", shorten >=0pt, shorten <=0pt] \\
				y \ast z \arrow[r,"\beta_{x\ast z, y \ast z}"',shorten >=-4pt, shorten <=-4pt] & 
				(x \ast y)\ast z
			\end{tikzcd}
		\end{center}
		\medskip
		\begin{center}
			\begin{tikzcd}[row sep=3em, column sep=5em]
				x \arrow[r, "\alpha_{x,y}", shorten >=2pt, shorten <=2pt] \arrow[d, "\mu_x"', shorten >=0pt, shorten <=0pt] & x \ast y \arrow[d, "\mu_{x \ast y}", shorten >=0pt, shorten <=0pt] \\
				\U_X(x) \arrow[r, "\alpha_{\U_X(x), y}"'] & 
				\U_X(x \ast y)
			\end{tikzcd}
			\hspace{2cm}
			\begin{tikzcd}[row sep=3em, column sep=5em]
				x \arrow[r, "\alpha_{x,y}", shorten >=2pt, shorten <=2pt] \arrow[d, "\partial_x"', shorten >=0pt, shorten <=0pt] & x \ast y \arrow[d, "\partial_{x \ast y}", shorten >=0pt, shorten <=0pt] \\
				\D_X(x) \arrow[r, "\alpha_{\D_X(x), y}"'] & 
				\D_X(x \ast y)
			\end{tikzcd}
		\end{center}
		\medskip
		
		\begin{center}
			\hspace{0.4 cm}
			\begin{tikzcd}[row sep=3em, column sep=6em]
				x \arrow[r, "\alpha_{x,\U_X(y)}", shorten >=2pt, shorten <=2pt] \arrow[d, "\id_x"', shorten >=0pt, shorten <=0pt] & x \ast \U_X(y) \arrow[d, "\id_{x \ast \U_X(y)}", shorten >=0pt, shorten <=0pt] \\
				x \arrow[r, "\alpha_{x, y}"'] & 
				x \ast y
			\end{tikzcd}
			\hspace{1.5 cm}
			\begin{tikzcd}[row sep=3em, column sep=6em]
				x \arrow[r, "\alpha_{x,\D_X(y)}", shorten >=2pt, shorten <=2pt] \arrow[d, "\id_x"', shorten >=0pt, shorten <=0pt] & x \ast \D_X(y) \arrow[d, "\id_{x \ast \D_X(y)}", shorten >=0pt, shorten <=0pt] \\
				x \arrow[r, "\alpha_{x, y}"'] & 
				x \ast y
			\end{tikzcd}
		\end{center}
		\medskip
		\begin{center}
			\hspace{-0.5cm}
			\begin{tikzcd}[row sep=3em, column sep=4.5em]
				y \arrow[r, "\beta_{x,y}", shorten >=2pt, shorten <=2pt] \arrow[d, "\beta_{\U_X(x),y}"', shorten >=0pt, shorten <=0pt] & x \ast y \arrow[d, "\mu_{x \ast y}", shorten >=0pt, shorten <=0pt] \\
				\U_X(x) \ast y \arrow[r, "\id_{\U_X(x \ast y)}"'] & 
				\U_X(x \ast y)
			\end{tikzcd}
			\hspace{1.7 cm}
			\begin{tikzcd}[row sep=3em, column sep=4.5em]
				y \arrow[r, "\beta_{x,y}", shorten >=2pt, shorten <=2pt] \arrow[d, "\beta_{\D_X(x),y}", shorten >=0pt, shorten <=0pt] & x \ast y \arrow[d, "\partial_{x \ast y}", shorten >=0pt, shorten <=0pt] \\
				\D_X(x) \ast y \arrow[r, "\id_{\D_X(x \ast y)}"'] & 
				\D_X(x \ast y)
			\end{tikzcd}
		\end{center}
		\medskip
		
		\begin{center}
			\hspace{0.5cm}
			\begin{tikzcd}[row sep=3em, column sep=5.5em]
				y \arrow[r, "\mu_{y}", shorten >=2pt, shorten <=2pt] \arrow[d, "\beta_{x,y}"', shorten >=0pt, shorten <=0pt] & \U_X(y) \arrow[d, "\beta_{x,\U_X(y)}", shorten >=0pt, shorten <=0pt] \\
				x \ast y \arrow[r, "\id_{x\ast y}"'] & 
				x \ast \U_X(y) 
			\end{tikzcd}
			\hspace{1.6 cm}
			\begin{tikzcd}[row sep=3em, column sep=5.5em]
				y \arrow[r, "\partial_{y}", shorten >=2pt, shorten <=2pt] \arrow[d, "\beta_{x,y}"', shorten >=0pt, shorten <=0pt] & \D_X(y) \arrow[d, "\beta_{x,\D_X(y)}", shorten >=0pt, shorten <=0pt] \\
				x \ast y \arrow[r, "\id_{x\ast y}"'] & 
				x \ast \D_X(y) 
			\end{tikzcd}
		\end{center}
		
	\end{itemize}
	
	Thus, a trunk map $A:\T(X, \ast,\U_X,\D_X) \rightarrow \Ab$ yields abelian groups $A_x$ and group homomorphisms $\phi_{x,y}:A_x \rightarrow A_{x \ast y}$, $\psi_{x,y}:A_y \rightarrow A_{x \ast y}$, $\zeta_x:A_x \rightarrow A_{\U_X(x)}$ and $\Omega_x:A_x \rightarrow A_{\D_X(x)}$ such that:
	\begin{enumerate}[(M1)]
		\item [(M1)] \label{M1} $\phi_{x \ast y,z}\phi_{x,y}=\phi_{x \ast z,y \ast z}\phi_{x,z}$,
		\item [(M2)] \label{M2} $\phi_{x \ast y,z}\psi_{x,y}=\psi_{x \ast z,y \ast z}\phi_{y,z}$,
		\item [(M3)] \label{M3} $\zeta_{x \ast y} \phi_{x,y}= \phi_{\U_X(x),y}\zeta_x$,
		\item [(M4)] \label{M4} $\Omega_{x \ast y} \phi_{x,y}= \phi_{\D_X(x),y}\Omega_x$,
		\item [(M5)] \label{M5} $\phi_{x,\U_X(y)}=\phi_{x,y}$,
		\item [(M6)]\label{M6} $\phi_{x,\D_X(y)}=\phi_{x,y}$,
		\item [(M7)]\label{M7} $\psi_{\U_X(x),y}=\zeta_{x \ast y}\psi_{x,y}$,
		\item [(M8)]\label{M8} $\psi_{\D_X(x),y}=\Omega_{x \ast y}\psi_{x,y}$,
		\item [(M9)]\label{M9}$\psi_{x,\U_X(y)} \zeta_{y}=\psi_{x,y}$,
		\item [(M10)]\label{M10}$\psi_{x,\D_X(y)} \Omega_{y}=\psi_{x,y}$,
	\end{enumerate}
	for all $x,y,z \in X$. We denote such a trunk map by $\mathscr{F}=(A,\phi,\psi,\zeta,\Omega)$.
	
	\begin{defn}
		Let $(X, \ast,\U_X,\D_X)$ be a $\GL$-rack. Then, an {\it $(X, \ast,\U_X,\D_X)$-module} is a trunk map $\mathscr{F}=(A,\phi,\psi,\zeta,\Omega): T(X, \ast,\U_X,\D_X) \rightarrow \Ab$ such that $\phi_{x,y}: A_{x}\rightarrow A_{x\ast y}$ is an isomorphism and 
		\begin{enumerate}
			\item[(M11)] \label{M11} $\psi_{x\ast y,z}(a)= \phi_{x\ast z,y\ast z}\psi_{x,z}(a)+\psi_{x\ast z,y\ast z}\psi_{y,z}(a)$,
			\item[(M12)] \label{M12} $\zeta_{\D_X(x\ast x)}(\Omega_{x \ast x}(\phi_{x,x}(b) + \psi_{x,x}(b)))=b $,
			\item[(M13)] \label{M13} $\Omega_{\U_X(x\ast x)}(\zeta_{x \ast x}(\phi_{x,x}(b) + \psi_{x,x}(b)))=b $
		\end{enumerate}
		hold for all $a \in A_z$, $b \in A_x$ and $x,y,z \in X$. 
		\par		
	\end{defn}
	
	\begin{defn}
		Let $(X, \ast,\U_X,\D_X)$ be a $\GL$-rack. An $(X, \ast,\U_X,\D_X)$-module  $\mathscr{F}=(A,\phi,\psi,\zeta,\Omega)$ is called \textit{homogeneous} if the constituent groups are all isomorphic, that is, $A_x \cong A_y$ for all $x,y \in X$.
	\end{defn}		
	
	Before proceeding further, let us examine some examples of modules over $\GL$-racks.
	
\begin{example}
Suppose that $(X, \ast,\U_X,\D_X)$ is a $\GL$-rack.
\begin{enumerate}
			\item Any abelian group $A$ can be considered as an $(X, \ast,\U_X,\D_X)$-module by taking $A_x=A$ $\phi_{x,y}=\id_A$, $\psi_{x,y}=0$, $\zeta_x=\id_A$ and $\Omega_x=\id_A$ for all $x,y \in X$, which we call a \textit{trivial homogeneous module}.
			\item \label{example-module} Consider a family of abelian groups $\{ G_{i}\}_{i\in I}$. Let $\alpha_i, \beta_i, \gamma_i \in \Aut(G_i)$ such that  $\alpha_{i}\beta_{i}\gamma_i=\id_{G_i}$ and $\alpha_{i},\, \beta_{i},\,\gamma_i$ commute with each other for all $i \in I$. Then, $\mathscr{F}=(A,\phi,\psi,\zeta,\Omega)$ is an $(X, \ast,\U_X,\D_X)$-module, where
			$A_x = \prod_{i\in I} G_i $, $\phi_{x,y}= \prod_{i \in I} \alpha_i$,  $\psi_{x,y} = 0$, $\zeta_{x}=\prod_{i \in I} \beta_i$  and $\Omega_{x}=\prod_{i \in I} \gamma_i$ for all $x \in X$.
			\item  Let $A$ be an abelian group. Then, $\mathscr{F}=(A,\phi,\psi,\eta)$ is an $(X, \ast,\U_X,\D_X)$-module, where $A_x = A$, $\phi_{x,y}(a)=-a$, $\psi_{x,y}(a) = 2a$, $\zeta_{x}(a)=a$ and $\Omega_{x}(a)=a$ for all $x \in X$ and $a \in A$.
		\end{enumerate}
	\end{example}
	
	\begin{defn}\label{map}
		Let $(X, \ast,\U_X,\D_X)$ be a $\GL$-rack, and $\mathscr{F}=(A,\phi,\psi,\zeta,\Omega)$ and $\mathscr{F'}=(A',\phi',\psi',\zeta',\Omega')$ be $(X, \ast,\U_X,\D_X)$-modules. An \textit{$(X, \ast,\U_X,\D_X)$-map} is a natural transformation $f: \mathscr{F}\rightarrow \mathscr{F}'$ of trunk maps, that is, it is a collection  $f=\{f_x: A_x\rightarrow A'_x \mid  x \in X\}$ of group homomorphisms such that
		\begin{eqnarray}
			\phi'_{x,y}f_x &=& f_{x\ast y} \phi_{x,y},\label{X-map-1}\\
			\psi'_{x,y}f_y &=& f_{x\ast y}\psi_{x,y},\label{X-map-2}\\
			\zeta'_{x}f_x &=& f_{\U_X(x)} \zeta_{x},\label{X-map-3}\\
			\Omega'_{x}f_x &=& f_{\D_X(x)} \Omega_{x}\label{X-map-4},
		\end{eqnarray}
		for all $x,y \in X$.
		\par
		In addition, if each $f_x$ is an isomorphism of groups, then we say that $f: \mathscr{F}\rightarrow \mathscr{F}'$  is an isomorphism of  $(X, \ast,\U_X,\D_X)$-modules. 
	\end{defn}
	
	\begin{remark}
		If $(X, \ast,\U_X,\D_X)$ is a $\GL$-rack, then we can form the category $\mathbf{GLMod}_{(X, \ast,\U_X,\D_X)}$, whose objects are $(X, \ast,\U_X,\D_X)$-modules and whose morphisms are $(X, \ast,\U_X,\D_X)$-maps. 
	\end{remark}
	\medskip
	

	\section{Category of modules over $\GL$-racks}\label{sec Category of modules over GL-racks}
	In this section, we describe the objects in the category $\Ab(\mathbf{GL}|_{(X, \ast,\U_X,\D_X)})$ of abelian group objects in the slice category $\mathbf{GL}|_{(X, \ast,\U_X,\D_X)}$ over a fixed $\GL$-rack $(X, \ast,\U_X,\D_X)$. Let $(X, \ast,\U_X,\D_X)$ be a $\GL$-rack and $\mathscr{F}=(A,\phi,\psi,\zeta,\Omega)$ an $(X, \ast,\U_X,\D_X)$-module. We define the \textit{semi-direct product} $\mathscr{F} \rtimes X$ of $\mathscr{F}$ and $X$ to be the set
	$$ \big\{(a,x) \mid x\in X, \,  a \in A_x \big\}$$ 
	equipped with the binary operation
	\begin{equation}
		(a,x) \tilde{\ast} (b,y):=\big(\phi_{x,y}(a)+\psi_{x,y}(b), \, x \ast y \big),
	\end{equation}
	and the maps $\U_{\mathscr{F} \rtimes X},\D_{\mathscr{F} \rtimes X}:\mathscr{F} \rtimes X \rightarrow \mathscr{F} \rtimes X$ are defined by 
	$$\U_{\mathscr{F} \rtimes X}\big((a,x) \big)= \big(\zeta_x(a), \U_X(x)\big)$$ and 	$$\D_{\mathscr{F} \rtimes X}\big((a,x) \big)= \big(\Omega_x(a), \D_X(x)\big)$$
	for all $a \in A_x$,  $b \in A_y$ and $x,y \in X$.
	
	\begin{prop}\label{semi-direct product}
		Let $(X, \ast,\U_X,\D_X)$ be a $\GL$-rack and $\mathscr{F}=(A,\phi,\psi,\zeta,\Omega)$ an $(X, \ast,\U_X,\D_X)$-module. Then the semi-direct product $({\mathscr{F} \rtimes X},\tilde{\ast},\U_{\mathscr{F} \rtimes X},\D_{\mathscr{F} \rtimes X})$ is a $\GL$-rack. 
	\end{prop}
	
	\begin{proof}
		Note that ${\mathscr{F} \rtimes X}$ is a rack, as proved in \cite[Proposition 2.1]{MR2155522}. It remains to check that the maps $\U_{\mathscr{F} \rtimes X}$ and $\D_{\mathscr{F} \rtimes X}$ indeed give rise to a $\GL$-rack structure on ${\mathscr{F} \rtimes X}$.
		\begin{itemize}
			\item For (L1), we see that $$\U_{\mathscr{F} \rtimes X}\D_{\mathscr{F} \rtimes X}((a,x)\tilde{\ast}(a,x))=(\zeta_{\D_X(x \ast x)}(\Omega_{x \ast x}(\phi_{x,x}(a)+\psi_{x,x}(a))), \,\U_X\D_X(x\ast x))=(a,x) ,$$ where the last equality follows from the axiom \hyperref[M12]{(M12)}.
			\item For (L2), we have,
			\begin{eqnarray*}
				\U_{\mathscr{F} \rtimes X}\big((a,x)\big) \tilde{\ast} (b,y)&=&(\zeta_x(a), \,\U_X(x)) \tilde{\ast}(b,y)\\
				&=&(\phi_{\U_X(x),y}(\zeta_x(a))+\psi_{\U_X(x),y}(b), \, \U_X(x) \ast y)\\
				&=& (\zeta_{x \ast y}(\phi_{x,y}(a)+\psi_{x,y}(b)),\, \U_X(x \ast y)),\\
				&& \textrm{by axioms \hyperref[M3]{(M3)} and \hyperref[M7]{(M7)}}\\
				&=& \U_{\mathscr{F} \rtimes X}(\phi_{x,y}(a)+\psi_{x,y}(b), \, x \ast y)\\
				&=& \U_{\mathscr{F} \rtimes X} \big((a,x) \tilde{\ast}(b,y) \big).
			\end{eqnarray*}
			\item For (L3), we have 
			\begin{eqnarray*}
				(a,x) \tilde{\ast} \U_{\mathscr{F} \rtimes X} \big((b,y)\big) &=& (a,x) \tilde{\ast} (\zeta_y(b), \,\U_X(y))\\
				&=& (\phi_{x, \U_X(y)}(a)+\psi_{x,\U_X(y)}(\zeta_y(b)),\, x \ast \U_X(y)),\\
				&=& (\phi_{x,y}(a)+\psi_{x,y}(b), \,x \ast y),\\
				&& \textrm{by axioms \hyperref[M5]{(M5)} and \hyperref[M9]{(M9)}}\\
				&=& (a,x)\tilde{\ast}(b,y).
			\end{eqnarray*}
		\end{itemize}
The verifications corresponding to (L1$^{\prime}$),(L2$^{\prime}$) and (L3$^{\prime}$) are similar to those of (L1),(L2) and (L3), respectively. Thus, the semi-direct product $({\mathscr{F} \rtimes X},\tilde{\ast}, \U_{\mathscr{F} \rtimes X},\D_{\mathscr{F} \rtimes X})$ is a $\GL$-rack. 
	\end{proof}
	
	\begin{prop}\label{abelian-group-object}
		Let $(X, \ast,\U_X,\D_X)$ be a $\GL$-rack and $\mathscr{F}=(A,\phi,\psi,\zeta,\Omega)$ an  $(X, \ast,\U_X,\D_X)$-module. Then there exists an abelian group object $p:({\mathscr{F} \rtimes X},\tilde{\ast},\U_{\mathscr{F} \rtimes X},\D_{\mathscr{F} \rtimes X}) \rightarrow (X, \ast,\U_X,\D_X)$, which we denote by $\mathcal{T}(\mathscr{F})$, in the slice category over $(X, \ast,\U_X,\D_X)$. Further, let  $\mathscr{F'}=(A',\phi',\psi',\eta')$ be another $(X, \ast,\U_X,\D_X)$-module and $f:\mathscr{F}\rightarrow \mathscr{F'}$ an $(X, \ast,\U_X,\D_X)$-map. Then the map $\mathcal{T}(f): \mathscr{F}\rtimes X \rightarrow \mathscr{F'}\rtimes X$ defined by $\mathcal{T}(f)\big((a,x) \big)= (f_{x}(a),x)$ gives a slice morphism. Moreover, $\mathcal{T}: \mathbf{GLMod}_{(X, \ast,\U_X,\D_X)} \rightarrow \Ab(\mathbf{GL}|_{(X, \ast,\U_X,\D_X)})$ is a functor.
	\end{prop}
	
\begin{proof}
Since $\mathscr{F}=(A,\phi,\psi,\zeta,\Omega)$ is an $(X, \ast,\U_X,\D_X)$-module,  by Proposition \ref{semi-direct product}, the semi-direct product $({\mathscr{F} \rtimes X},\tilde{\ast},\U_{\mathscr{F} \rtimes X},\D_{\mathscr{F} \rtimes X})$ is a $\GL$-rack. Let $p: \mathscr{F} \rtimes X \rightarrow X$ be the natural projection given as $p(a,x)= x$ for $a \in A_x$ and $x \in X$. Then $p$ is a homomorphism of $\GL$-racks. We claim that $p: \mathscr{F} \rtimes X \rightarrow X$ has the canonical structure of an abelian group object. Let $$\big((\mathscr{F} \rtimes X )\times_X (\mathscr{F} \rtimes X),\star, \U_{(\mathscr{F} \rtimes X )\times_X (\mathscr{F} \rtimes X)},\D_{(\mathscr{F} \rtimes X )\times_X (\mathscr{F} \rtimes X)}\big) \rightarrow (X, \ast,\U_X,\D_X)$$ be the categorical product of $p: \mathscr{F} \rtimes X \rightarrow X$ with itself in the slice category over $(X, \ast,\U_X,\D_X)$. The $\GL$-rack structure $\star$ is defined in the same way as in the cartesian product. We define the morphisms
		\begin{itemize}
			\item $\mu:(\mathscr{F} \rtimes X )\times_X (\mathscr{F} \rtimes X) \rightarrow (\mathscr{F} \rtimes X) $ given by $((a_1,x),(a_2,x)) \mapsto (a_1+a_2,x)$,
			\item $\sigma:X \rightarrow (\mathscr{F} \rtimes X)$ given by $x \mapsto (0,x)$,
			\item $\nu:(\mathscr{F} \rtimes X) \rightarrow (\mathscr{F} \rtimes X)$ given by $(a,x) \mapsto (-a,x)$,
		\end{itemize}
		for all $a_1,a_2, a \in A_x$ and $x \in X$. It easy to check that $\mu$, $\sigma$ and $\nu$ are $\GL$-rack homomorphisms. Also, the fact that $\mu$, $\sigma$ and $\nu$ are slice morphisms and that they turn $p: \mathscr{F} \rtimes X \rightarrow X$ into an abelian group object follows directly (see also \cite[Theorem 2.2]{MR2155522}).
		\par	
		
		Let us define a functor $\mathcal{T}: \mathbf{GLMod}_{(X, \ast,\U_X,\D_X)} \rightarrow \Ab(\mathbf{GL}|_{(X, \ast,\U_X,\D_X)})$.  Given an $(X, \ast,\U_X,\D_X)$-module $\mathscr{F}=(A,\phi,\psi,\zeta,\Omega)$, let $\mathcal{T}(\mathscr{F})$ denote the object $p:\mathscr{F}\rtimes X\rightarrow X$. Let $\mathscr{F'}=(A',\phi',\psi',\zeta',\Omega')$ be another $(X, \ast,\U_X,\D_X)$-module and $f:\mathscr{F}\rightarrow \mathscr{F'}$ an $(X, \ast,\U_X,\D_X)$-map. As defined earlier, let $p': \mathscr{F} \rtimes X \rightarrow X$ be the natural projection given as $p'(a,x)= x$ for $a \in A^{'}_x$ and $x \in X$. Define $\mathcal{T}(f):\mathscr{F}\rtimes X \rightarrow \mathscr{F'}\rtimes X$ by $$T(f)(a,x)=(f_{x}(a),x).$$ Then, we have
		\begin{eqnarray*}
			\mathcal{T}(f) \big((a,x) \tilde{\ast} (b,y) \big)&=&\mathcal{T}(f)\big(\phi_{x,y}(a)+\psi_{x,y}(b), \,x \ast y \big)\\
			&=& \big(f_{x \ast y}(\phi_{x,y}(a)+\psi_{x,y}(b)),  \,x \ast y \big)\\
			&=&\big(\phi'_{x,y}f_x(a)+\psi'_{x,y}f_y(b),  \,x \ast y \big), \quad \textrm{since $f$ is an $(X, \ast,\U_X,\D_X)$-map}\\
			&=& \mathcal{T}(f)\big((a,x)\big) \star' \mathcal{T}(f)\big((b,y)\big),
		\end{eqnarray*}
		where $\tilde{\ast}$ is the rack operation in $\mathscr{F}\rtimes X$ and $\tilde{\star}$ is the rack operation in $\mathscr{F'}\rtimes X$. Further, since $\zeta'_{x}f_x = f_{\U_X(x)} \zeta_{x}$ and $\Omega'_{x}f_x = f_{\D_X(x)} \Omega_{x}$, the  diagrams
		
		\begin{center}
			\begin{tikzcd}[row sep=3em, column sep=4em]
				\mathscr{F}\rtimes X \arrow[r, "\U_{\mathscr{F}\rtimes X}"] \arrow[d, "\mathcal{T}(f)"] & \mathscr{F}\rtimes X \arrow[d, "\mathcal{T}(f)"] \\
				\mathscr{F'}\rtimes X \arrow[r, "\U_{\mathscr{F'}\rtimes X}"] & \mathscr{F'}\rtimes X
			\end{tikzcd}
			\hspace{1.5 cm}
			\begin{tikzcd}[row sep=3em, column sep=4em]
				\mathscr{F}\rtimes X \arrow[r, "\D_{\mathscr{F}\rtimes X}"] \arrow[d, "\mathcal{T}(f)"] & \mathscr{F}\rtimes X \arrow[d, "\mathcal{T}(f)"] \\
				\mathscr{F'}\rtimes X \arrow[r, "\D_{\mathscr{F'}\rtimes X}"] & \mathscr{F'}\rtimes X
			\end{tikzcd}
		\end{center}
		
		must commute. Thus,  $T(f)$ is a homomorphism of $\GL$-racks. Finally, since the diagram
		\begin{center}
			\begin{tikzcd}
				\mathscr{F}\rtimes X \arrow[r, "\mathcal{T}(f)"] \arrow[rd, "p"'] & \mathscr{F'}\rtimes X \arrow[d, "p'"] \\
				& X
			\end{tikzcd}
		\end{center}
		commutes, $\mathcal{T}(f)$ is a slice morphism. In fact, one can check that $\mathcal{T}$ is a functor (see \cite[Theorem 2.2]{MR2155522}).
\end{proof}

\begin{prop}\label{module}
Let $(X, \ast,\U_X,\D_X)$ be a $\GL$-rack, and $p:(Y,\star,\U_Y, \D_Y) \rightarrow (X, \ast,\U_X,\D_X)$ an abelian group object with the multiplication map $m$, the inverse map $i$, and the section $s$. Let $(X, *)$ and $(Y, \star)$ be the underlying racks. Then there exists an $(X, \ast,\U_X,\D_X)$-module $(R, \phi, \psi, \zeta, \Omega)$ given by an induced abelian group structure on the fibre $R_x:=p^{-1}(x)$ for each $x \in X$, where
		$$\phi_{x,y}:R_x \rightarrow R_{x \ast y} ~~ \text{is given by }~ u \mapsto u \star s(y),$$  
		$$\psi_{x,y}:R_y \rightarrow R_{x \ast y} ~~ \text{is given by }~ v \mapsto s(x) \star v,$$  
		$$\zeta_x: R_x \rightarrow R_{\U_Y(x)} ~~ \text{is given by }~ t \mapsto \U_{Y}(t)$$ and
		$$\Omega_x: R_x \rightarrow R_{\D_Y(x)} ~~ \text{is given by }~ t \mapsto \D_{Y}(t)$$
		for all $t,u\in R_x$, $v \in R_y$ and $x,y \in X$. Furthermore, the association gives a functor $\mathcal{S}: \Ab(\mathbf{GL}|_{(X, \ast,\U_X,\D_X)}) \to  \mathbf{GLMod}_{(X, \rho_X)}$. 
	\end{prop}
	
	\begin{proof}
		Let $p:(Y,\star,\U_Y, \D_Y) \rightarrow (X, \ast,\U_X,\D_X)$ be an abelian group object in $\Ab(\mathbf{GL}|_{(X, \ast,\U_X,\D_X)})$. Let $R_x=p^{-1}(x)$ for each $x \in X$. Then $R_x$ admits an abelian group structure given by $$u+v:=m(u,v),$$ where $s(x)$ is the identity element in $R_x$ and $-u:=i(u)$ for all $u,v \in R_x$. To define an $(X, \ast,\U_X,\D_X)$-module structure, let us define
		$$\phi_{x,y}:R_x \rightarrow R_{x \ast y} ~~\text{as} ~u \mapsto u \star s(y),$$  
		$$\psi_{x,y}:R_y \rightarrow R_{x \ast y} ~~\text{as}~ v \mapsto s(x) \star v,$$  
		$$\zeta_x: R_x \rightarrow R_{\U_X(x)} ~~ \text{as}~ t \mapsto \U_{Y}(t)$$ and
		$$\Omega_x: R_x \rightarrow R_{\D_X(x)} ~~ \text{as}~ t \mapsto \D_{Y}(t)$$
		for all $t,u\in R_x$, $v \in R_y$ and $x,y \in X$.
		It has been proved in  \cite[Theorem 2.2]{MR2155522} that $\phi_{x,y}$ and $\psi_{x,y}$ are group homomorphisms for all $x, y \in X$. For the case of $\zeta_x$, if $t_1,t_2 \in R_x$, then we see that
		\begin{eqnarray*}
			\zeta_x(t_1+t_2)&=& \U_Y(m(t_1,t_2))\\
			&=& m(\U_Y(t_1), \U_Y(t_2)), \quad \textrm{since $m \,(\U_Y \times \U_Y)=\U_Y \, m$}\\
			&=& \U_Y(t_1)+\U_Y(t_2)\\
			&=& \zeta_x(t_1)+\zeta_x(t_2).
		\end{eqnarray*}
		The verfication for $\Omega_x$ follows similarly. Next, we need to check that $\phi_{x,y}$, $\psi_{x,y}$, $\zeta_x$ and $\Omega_x$ satisfy the $(X, \ast,\U_X,\D_X)$-module identities. Since \hyperref[M1]{(M1)}, \hyperref[M2]{(M2)} and \hyperref[M11]{(M11)} are proved in \cite[Theorem 2.2]{MR2155522}, we verify only the remaining identities.
		\begin{itemize}
			
			\item Using $\GL$-rack axiom (L2), we have $$\phi_{\U_X(x),y}(\zeta_x(t))=\phi_{\U_X(x),y}(\U_Y(t))=\U_Y(t)\star s(y)=\U_Y(t \star s(y))=\zeta_{x \ast y}(t \star s(y))=\zeta_{x \ast y} (\phi_{x,y}(t)),$$ and hence \hyperref[M3]{(M3)} holds.
			\item Using the $\GL$-rack axiom (L3) and the fact that $s$ is a homomorphism of $\GL$-racks, we see that $$ \phi_{x,\U_X{(y)}}(t)=t \star s(\U_X(y))=t \star \U_Y(s(y))=t\star s(y)=\phi_{x,y}(t),$$ and hence \hyperref[M5]{(M5)} holds.
			\item Using the $\GL$-rack axiom (L2) and the fact that $s$ is a homomorphism of $\GL$-racks, we see that $$\psi_{\U_X(x),y}(v)= s(\U_X(x))\star v=\U_Y(s(x)) \star v=\U_Y(s(x)\star v)=\zeta_{x \ast y}(s(x) \star v)=\zeta_{x \ast y}(\psi_{x,y}(v)),$$ and hence \hyperref[M7]{(M7)} holds.
			\item Using the $\GL$-rack axiom (L3) we see that,
			$$\psi_{x,\U_X(y)}\zeta_y(v)=\psi_{x,\U_X(y)}\U_Y(v)=s(x)\star \U_Y(v)= s(x)\star v = \psi_{x,y}(v),$$ and hence \hyperref[M9]{(M9)} holds.
			
			\item  Note that, 
			\begin{eqnarray*}
				\zeta_{\D_X(x*x)}(\Omega_{x*x}(\phi_{x,x}(v)+\psi_{x,x}(v)))&=&\zeta_{\D_X(x*x)}(\Omega_{x*x}(v \star s(x)+s (x) \star v))\\
				&=& \U_Y \D_Y(v\star s(x)+s(x)\star v)\\
				&=& \U_Y \D_Y(m(v\star s(x),s(x)\star v))\\
				&=& \U_Y \D_Y(m(v, s(x))\tilde{\star}(s(x), v))\\
				&=& \U_Y \D_Y(m(v, s(x))\star m(s(x), v))\\
				&=& \U_Y \D_Y(v \star v)\\
				&=& v
			\end{eqnarray*} and hence \hyperref[M12]{(M12)} holds.
		\end{itemize}
		It is easy to see that \hyperref[M4]{(M4)}, \hyperref[M6]{(M6)}, \hyperref[M8]{(M8)}, \hyperref[M10]{(M10)}, and \hyperref[M13]{(M13)} hold for reasons same as in \hyperref[M3]{(M3)}, \hyperref[M5]{(M5)}, \hyperref[M7]{(M7)}, \hyperref[M9]{(M9)}, and \hyperref[M12]{(M12)}, respectively. Hence, we have proved that $(R, \phi, \psi, \zeta, \Omega)$ is an $(X, \ast,\U_X,\D_X)$-module.
		\par 
		
		Next, we construct a functor $\mathcal{S}: \Ab(\mathbf{GL}|_{(X, \ast,\U_X,\D_X)}) \to  \mathbf{GLMod}_{(X, \ast,\U_X,\D_X)}$. Let $\mathcal{O}$ denote the abelian group object $p:(Y,\star,\U_Y, \D_Y) \rightarrow (X, \ast,\U_X,\D_X)$. We set $\mathcal{S}(\mathcal{O}) = (R,\phi,\psi,\zeta, \Omega)$, which is defined above. Let $f$ be a morphism from the abelian group object $p_1:(Y_1,\star_1,\U_{Y_1}, \D_{Y_1}) \rightarrow (X, \ast,\U_X,\D_X)$ to the abelian group object $p_2:(Y_2,\star_2,\U_{Y_2}, \D_{Y_2}) \rightarrow (X, \ast,\U_X,\D_X)$. Thus, if $m_1$ and $m_2$ denote the corresponding multiplication maps, then $f \, m_1= m_2 \,(f, f)$. Further, $f:(Y_1,\star_1,\U_{Y_1}, \D_{Y_1}) \rightarrow (Y_2,\star_2,\U_{Y_2}, \D_{Y_2})$ is a homomorphism of $\GL$-racks with $p_2\,f=p_1$. Let $\mathscr{F}_1=(R_1,\phi,\psi,\zeta,\Omega)$ and $\mathscr{F}_2=(R_2,\phi',\psi',\zeta',\Omega')$ denote the $(X, \ast,\U_X,\D_X)$-modules as described above. Define $g:\mathscr{F}_1 \rightarrow \mathscr{F}_2$ by setting
		$$g_x:{(R_1)}_x \rightarrow {(R_2)}_x~~\text{as}~~g_x(u)= f(u)$$
		for all $x\in X$ and $u \in {(R_1)}_x$. Since   $f \, m_1= m_2 \,(f, f)$, it follows that $g_x$ is an abelian group homomorphism for each $x \in X$. The identities  
		$\phi'_{x,y}g_x = g_{x\ast y} \phi_{x,y}$ and  $\psi'_{x,y}g_y = g_{x\ast y}\psi_{x,y}$ are  proved in \cite[Theorem 2.2]{MR2155522}. Further, we have 
		$$g_{\U_X(x)} \zeta_{x}(w)=g_{\U_X(x)} (\U_{Y_1}(w))=f(\U_{Y_1}(w))=\U_{Y_2}(f(w))=\U_{Y_2}g_x(w)=\zeta_x'(g_x(w))$$
		for all $w\in (R_1)_x$, and similarly it follows  that  $g_{\D_X(x)} \Omega_{x} = \Omega'_{x}g_x $. Hence, $g$ is an $(X, \ast,\U_X,\D_X)$-map. A routine check shows that $\mathcal{S}$ is a functor (see \cite[Thereom 2.2]{MR2155522}).
	\end{proof}
	
We now present the main result of this section.
	
	\begin{theorem}\label{equivalence of categories for symmetric rack modules}
		Let $(X, \ast,\U_X,\D_X)$ be a $\GL$-rack. Then the category $\mathbf{GLMod}_{(X, \ast,\U_X,\D_X)}$ of $(X, \ast,\U_X,\D_X)$-modules is equivalent to the category $\Ab(\mathbf{GL}|_{(X, \ast,\U_X,\D_X)})$ of abelian group objects in the slice category over $(X, \ast,\U_X,\D_X)$.
	\end{theorem}
	
	\begin{proof} Let $(X, \ast,\U_X,\D_X)$ be a $\GL$-rack. By Proposition \ref{abelian-group-object} and Proposition \ref{module}, we have functors $\mathcal{T}:  \mathbf{GLMod}_{(X, \ast,\U_X,\D_X)} \to \Ab(\mathbf{GL}|_{(X, \ast,\U_X,\D_X)})$ and $\mathcal{S}: \Ab(\mathbf{GL}|_{(X, \ast,\U_X,\D_X)}) \to \mathbf{GLMod}_{(X, \ast,\U_X,\D_X)}$. We show that $\mathcal{T}\mathcal{S}$ is naturally isomorphic to $\id_{\Ab(\mathbf{GL}|_{(X, \ast,\U_X,\D_X)})}$, whereas $\mathcal{S}\mathcal{T}$ is naturally isomorphic to $\id_{ \mathbf{GLMod}_{(X, \ast,\U_X,\D_X)}}$.
		\par
		
		Let $\mathcal{O} \in  \Ab(\mathbf{GL}|_{(X, \ast,\U_X,\D_X)})$ denote the object $p:(Y, \star,\U_Y,\D_Y) \rightarrow (X, \ast,\U_X,\D_X)$ with the multiplication map $m$,  the inverse map $i$, and the section $s$. Then $\mathcal{T}\mathcal{S} (\mathcal{O})$ is the object $q:(\mathscr{R} \rtimes X, \tilde{*},\U_{\mathscr{R} \rtimes X},\D_{\mathscr{R} \rtimes X}) \rightarrow (X, \ast,\U_X,\D_X)$ in the category $ \Ab(\mathbf{GL}|_{(X, \ast,\U_X,\D_X)})$, where $\mathscr{R}=(R, \phi,\psi,\omega,\zeta)$ is as constructed in Proposition \ref{module}.  Note that, the abelian group structure on $q:(\mathscr{R} \rtimes X, \tilde{*},\U_{\mathscr{R} \rtimes X},\D_{\mathscr{R} \rtimes X}) \rightarrow (X, \ast,\U_X,\D_X)$ has the multiplication, the inverse and the section given by
		\begin{itemize}
			\item $\mu ((a_1,x),(a_2,x)) = (a_1+a_2,x)=(m(a_1,a_2),x)$,
			\item $\sigma ((a,x)) = (0,x)=(s(a),x)$,
			\item $\nu((a,x)) = (-a,x)=(i(a),x)$,
		\end{itemize}
		for all $a_1,a_2, a \in R_x$ and $x \in X$. Define $\Lambda_{\mathcal{O}}: (Y, \star,\U_Y,\D_Y) \rightarrow (\mathscr{R} \rtimes X, \tilde{*},\U_{\mathscr{R} \rtimes X},\D_{\mathscr{R} \rtimes X})$ by setting $\Lambda_{\mathcal{O}}(a)=(a,x)$ for  $a \in R_x$. Clearly, $\Lambda_{\mathcal{O}}$ is a bijection and the diagram
		\begin{center}
			\begin{tikzcd}
				(Y, \star,\U_Y,\D_Y) \arrow[r, "\Lambda_{\mathcal{O}}"] \arrow[rd, "p"'] & (\mathscr{R} \rtimes X, \tilde{*},\U_{\mathscr{R} \rtimes X},\D_{\mathscr{R} \rtimes X}) \arrow[d, "q"] \\
				& (X, \ast,\U_X,\D_X)
			\end{tikzcd}
		\end{center}
		commutes, where $q$ is the projection onto the second coordinate. Note that, $\Lambda_{\mathcal{O}}(m(a,b))=\mu(\Lambda_{\mathcal{O}}(a)$, $\Lambda_{\mathcal{O}}(b)),~ \Lambda_{\mathcal{O}}(s(a))=\sigma(\Lambda_{\mathcal{O}}(a))$, and $\Lambda_{\mathcal{O}}(i(a))=\nu(\Lambda_{\mathcal{O}}(a))$ for all $a,b \in R_x$. Hence, $\Lambda_{\mathcal{O}}$ respects the abelian group structure. Now, for $a \in R_x$ and $b \in R_y$, we have
		\begin{eqnarray*}
			\Lambda_{\mathcal{O}}(a) \tilde{\ast} \Lambda_{\mathcal{O}}(b) &=& (a,x) \tilde{\ast} (b,y) \\
			&=& \big(\phi_{x,y}(a)+\psi_{x,y}(b), \,x \ast y \big)\\
			&=& \big(a \star s(y)+ s(x) \star b, \,x \ast y \big)\\
			&=& \big(m \big(a \star s(y),s(x) \star b \big), \,x \ast y \big)\\
			&=& \big(m((a,s(x)) \tilde{\star} (s(y),b)), \,x \ast y\big)\\
			&=&\big(m(a,s(x)) \star m(s(y),b), x \ast y\big),\quad \textrm{since $m: Y \times Y \to Y$ is a rack homomorphism}\\
			&=& (a \star b,x \ast y)\\
			&=& \Lambda_{\mathcal{O}}(a \star b),
		\end{eqnarray*}
		where $\tilde{\star}$ is the product rack operation on $Y \times Y$. Thus, $\Lambda_{\mathcal{O}}$ is a rack homomorphism. Since $\U_{\mathscr{R}\rtimes X} \, \Lambda_{\mathcal{O}}=\Lambda_{\mathcal{O}} \, \U_Y$ and $\D_{\mathscr{R}\rtimes X} \, \Lambda_{\mathcal{O}}=\Lambda_{\mathcal{O}} \, \D_Y$, it follows that $\Lambda_{\mathcal{O}}$ is a $\GL$-rack isomorphism. Hence, 
		$\Lambda_{\mathcal{O}}: \mathcal{O} \to \mathcal{T}\mathcal{S} (\mathcal{O})$ is an isomorphism in $\Ab(\mathbf{GL}|_{(X, \ast,\U_X,\D_X)})$.
		\par
		
		Let $\mathcal{O}$ and $\mathcal{O'}$ denote the objects $p:(Y, \star,\U_Y,\D_Y) \rightarrow (X, \ast,\U_X,\D_X)$ and $p':(Y', \star',\U_Y',\D_Y') \rightarrow (X, \ast,\U_X,\D_X)$ in $\Ab(\mathbf{GL}|_{(X, \ast,\U_X,\D_X)})$, respectively. Let $f:\mathcal{O} \rightarrow \mathcal{O'}$ be a morphism. Then, by definition, $f: (Y, \star,\U_Y,\D_Y) \to (Y', \star',\U_Y',\D_Y')$ is a $\GL$-rack homomorphism such that $p' \, f=p$ and $f$ preserves the abelian group structure. Let $x \in X$ and $a \in R_x$. Then $\Lambda_{\mathcal{O'}} \, f(a)=(f_x(a),x)$, where $f_x=f|_{R_x}$ and $$\mathcal{T} \mathcal{S}(f) \, \Lambda_{\mathcal{O}} (a)=\mathcal{T} \mathcal{S}(f)(a,x)= \mathcal{T}(f|_{R_x}(a))=\mathcal{T}(f_x(a))=(f_x(a),x).$$ Thus, the diagram
		\begin{center}
			\begin{tikzcd}[row sep=3em, column sep=4em]
				\mathcal{O} \arrow[r, "\Lambda_{\mathcal{O}}"] \arrow[d, "f"] & \mathcal{T}\mathcal{S}(\mathcal{O}) \arrow[d, "\mathcal{T}\mathcal{S}(f)"] \\
				\mathcal{O'} \arrow[r, "\Lambda_{\mathcal{O'}}"] & \mathcal{T}\mathcal{S}(\mathcal{O'})
			\end{tikzcd}
		\end{center}
		commutes, and $\mathcal{T}\mathcal{S}$ is naturally isomorphic to $\id_{\Ab(\mathbf{GL}|_{(X, \ast,\U_X,\D_X)})}$.
		\par
		
		Next, we prove that $\mathcal{S}\mathcal{T}$ is naturally isomorphic to $\id_{\mathbf{GLMod}_{(X, \ast,\U_X,\D_X)}}$. Let $\mathcal{O}$ be the object $\mathscr{F}=(A,\phi,\psi,\zeta,\Omega)$ in $\mathbf{GLMod}_{(X, \ast,\U_X,\D_X)}$. Then $\mathcal{T}(\mathcal{O})$ is the abelian group object $p:\mathscr{F} \rtimes X \rightarrow X$ as in Proposition \ref{abelian-group-object}, with prescribed $\mu$, $\sigma$ and $\nu$. For each $x \in X$, since $A'_{x}:=p^{-1}(x)=(A_x,x)$, it follows that $\mathcal{S}\mathcal{T}(\mathcal{O}) $ is the object $\mathscr{F'}:=(A',\phi',\psi',\zeta',\Omega')$, where
		\begin{align}
			\phi_{x,y}' (a,x) &= (a,x) \tilde{\ast} \sigma(y)=\ (a,x) \tilde{\ast} (0,y)=(\phi_{x,y}(a), \,x \ast y),  \label{I}\\
			\psi_{x,y}' (b,y) &= \sigma(x) \tilde{\ast} (b,y)=(0,x) \tilde{\ast} (b,y) =(\psi_{x,y}(b), \,x \ast y), \label{II}\\ 
			\zeta_x' (a,x)& = \U_{\mathscr{F} \rtimes X}(a,x)=(\zeta_x(a), \, \U_X(x)), \label{III}\\
			\Omega_x' (a,x)& = \D_{\mathscr{F} \rtimes X}(a,x)=(\Omega_x(a), \, \D_X(x)), \label{IV}
		\end{align}	
		for all $a\in A_x$, $b \in A_y$ and $x,y \in X$. 
		\par
		
		We now construct an $(X, \ast,\U_X,\D_X)$-module isomorphism $\lambda_{\mathcal{O}}: \mathcal{O} \to \mathcal{S}\mathcal{T}(\mathcal{O})$. For each $ x \in X$, define $(\lambda_{\mathcal{O}})_x:A_x \to A'_x$ as $a \mapsto (a,x)$ for all $a \in A_x$. Clearly, $(\lambda_{\mathcal{O}})_{x}$ is an abelian group isomorphism for each $x \in X$. Also, using \eqref{I}, \eqref{II}, \eqref{III} and \eqref{IV}, we have
		\begin{eqnarray*}
			\phi'_{x,y}(\lambda_{\mathcal{O}})_{x}(a) &=& \phi'_{x,y}(a,x)= (\phi_{x,y}(a), \,x \ast y) =(\lambda_{\mathcal{O}})_{x\ast y} \phi_{x,y}(a),\\
			\psi'_{x,y}(\lambda_{\mathcal{O}})_{y}(b)&=&\psi'_{x,y}(b,y) = (\psi_{x,y}(b), \,x \ast y)= (\lambda_{\mathcal{O}})_{x\ast y}\psi_{x,y}(b),\\
			\zeta'_{x}(\lambda_{\mathcal{O}})_{x}(a) &=& \zeta'_x(a,x)=(\zeta_x(a), \,\U_X(x))=(\lambda_{\mathcal{O}})_{\U_X(x)} \zeta_{x}(a),\\
			\Omega'_{x}(\lambda_{\mathcal{O}})_{x}(a) &=& \Omega'_x(a,x)=(\Omega_x(a), \,\D_X(x))=(\lambda_{\mathcal{O}})_{\D_X(x)} \Omega_{x}(a),
		\end{eqnarray*}
		for all $a\in A_x$, $b \in A_y$ and $x,y \in X$.  Thus, \hyperref[X-map]{\eqref{X-map-1}-\eqref{X-map-4}} hold, and $\lambda_{\mathcal{O}}$ is an $(X, \ast,\U_X,\D_X)$-module isomorphism. Finally, let $\mathcal{O}_1$ and $\mathcal{O}_2$ be the objects  $\mathscr{F}_1=(A_1,\phi_1,\psi_1,\zeta_1,\Omega_1)$ and  $\mathscr{F}_2=(A_2,\phi_2,\psi_2,\zeta_2,\Omega_2)$ in $\mathbf{GLMod}_{(X, \ast,\U_X,\D_X)}$, respectively.  Let $g:\mathcal{O}_1 \rightarrow \mathcal{O}_2$ be an $(X, \ast,\U_X,\D_X)$-map. It is easy to see that the diagram
		\begin{center}
			\begin{tikzcd}[row sep=3em, column sep=4em]
				\mathcal{O}_1 \arrow[r, "\lambda_{\mathcal{O}_1}"] \arrow[d, "g"] & \mathcal{S}\mathcal{T}(\mathcal{O}_1)  \arrow[d, "\mathcal{S}\mathcal{T}(g)"] \\
				\mathcal{O}_2\arrow[r, "\lambda_{\mathcal{O}_2}"] & \mathcal{S}\mathcal{T}(\mathcal{O}_2)	\end{tikzcd}
		\end{center}
		commutes.  Thus, the functor $\mathcal{S}\mathcal{T}$ is naturally isomorphic to $\id_{\mathbf{GLMod}_{(X, \ast,\U_X,\D_X)}}$. This proves that the categories $\Ab(\mathbf{GL}|_{(X, \ast,\U_X,\D_X)})$ and $\mathbf{GLMod}_{(X, \ast,\U_X,\D_X)}$ are equivalent.
	\end{proof}
\medskip
	\ack{BK thanks CSIR for the PhD research fellowship. DS thanks IISER Mohali for the PhD research fellowship. MS is supported by the SwarnaJayanti Fellowship grants DST/SJF/MSA-02/2018-19 and SB/SJF/2019-20/04.}

\end{document}